\documentclass[a4paper,10pt,reqno]{amsart}
\usepackage{amssymb,latexsym,bbm,amsmath,amsfonts}
\usepackage{amsthm}
\usepackage[mathscr]{eucal}
\usepackage{rotating}
\usepackage{multirow}
\usepackage{slashbox}

\numberwithin{equation}{section}
\newtheorem{theorem}{Theorem}
\newtheorem{lemma}{Lemma}
\newtheorem{proposition}{Proposition}

\theoremstyle{definition}

\theoremstyle{remark}

\begin{document}

\title[Asymptotic Classification of SDEs with State--Independent Noise]
{Classification of the Asymptotic Behaviour of Globally Stable
Differential Equations with Respect to State--independent
Stochastic Perturbations}

\author{John A. D. Appleby}
\address{Edgeworth Centre for Financial Mathematics, School of Mathematical
Sciences, Dublin City University, Glasnevin, Dublin 9, Ireland}
\email{john.appleby@dcu.ie} \urladdr{webpages.dcu.ie/\textasciitilde
applebyj}

\author{Jian Cheng}
\address{Edgeworth Centre for Financial Mathematics, School of Mathematical
Sciences, Dublin City University, Glasnevin, Dublin 9, Ireland}
\email{jian.cheng2@mail.dcu.ie}

\author{Alexandra Rodkina}
\address{The University of the West Indies, Mona Campus
Department of Mathematics and Computer Science Mona, Kingston 7,
Jamaica} \email{alexandra.rodkina@uwimona.edu.jm}

\thanks{The first two authors gratefully acknowledges Science Foundation Ireland for the support of this research under the Mathematics Initiative 2007 grant 07/MI/008 ``Edgeworth Centre for Financial
Mathematics''.} 
\date{4 November 2012}

\begin{abstract}
In this paper we consider the global stability of solutions of a nonlinear
stochastic differential equation. The differential equation
is a perturbed version of a globally stable linear autonomous
equation with unique zero equilibrium where the diffusion
coefficient is independent of the state. Contingent on a dissipative condition
characterising the asymptotic stability of the unperturbed equation,
necessary and sufficient conditions on the rate of decay of the noise intensity
for the solution of the equation to be a.s. globally asymptotically
stable, contingent on some weak and noise independent reversion towards the equilibrium
when the solution is far from equilibrium. Under a stronger equilibrium reverting condition,
we may classify whether the solution globally asymptotically
stable, stable but not asymptotically stable, and unstable, each
with probability one purely in terms of the asymptotic intensity of the noise.
Sufficient conditions guaranteeing the different types of
asymptotic behaviour which are more readily checked are developed.
\end{abstract}

\maketitle

\section{Introduction}
In this paper, we characterise the stability, boundedness and instability of
the unique and globally stable equilibrium of a deterministic ordinary differential
equation when it is subjected to a stochastic perturbation
independent of the state. More specifically, we consider the asymptotic behaviour of solutions
of the $d$--dimensional stochastic differential equation
\begin{equation}  \label{eq.intromainstocheqn}
dX(t)=-f(X(t))\,dt + \sigma(t)\,dB(t),\quad t\geq 0; \quad
X(0)=\xi\in \mathbb{R}^d
\end{equation}
where $B$ is an $r$--dimensional standard Brownian motion,
$f:\mathbb{R}^d\to \mathbb{R}^d$ is a continuous
function and $\sigma\in C([0,\infty);\mathbb{R}^{d\times r})$, the continuity 
guaranteeing the existence of local solutions of \eqref{eq.intromainstocheqn}. 
There is no loss of generality in assuming that the unique equilibrium be at 0, so 
the equation without a stochastic perturbation is therefore
\begin{equation}\label{eq.unperturbed}
x'(t)=-f(x(t)), \quad t>0; \quad x(0)=\xi,
\end{equation}
and in order to guarantee the globally asymptotic stability we require that
 \begin{equation} \label{eq.xglobalstable}
\lim_{t\to\infty} x(t;\xi)=0 \text{ for all $\xi\in \mathbb{R}^d$}.
\end{equation}
This implies that $f(x)=0$ if and only if $x=0$. To characterise global asymptotic stability even for \eqref{eq.unperturbed} is difficult, so in general deterministic research has focussed on giving sufficient conditions under which all solutions of \eqref{eq.unperturbed}
obey $x(t)\to 0$ as $t\to\infty$. A popular assumption in the stochastic literature is the so called \emph{dissipative condition} 
\begin{equation} \label{eq.introfglobalunperturbed}
\langle x,f(x) \rangle>0 \quad \text{for all $x\neq 0$}.
\end{equation}
The dissipative condition ensures that $x=0$ is the unique equilibrium, for if there were another at $x^\ast\neq 0$, then we have
$0=\langle x^\ast,0\rangle = \langle x^\ast, f(x^\ast) \rangle >0$, a contradiction. We see also that in the one--dimensional deterministic 
case, the condition $xf(x)>0$ for $x\neq 0$, which characterises the existence of a unique and globally stable equilibrium, is nothing other than the dissipative
condition \eqref{eq.introfglobalunperturbed}. The proof that \eqref{eq.introfglobalunperturbed} implies \eqref{eq.xglobalstable} simply involves 
showing that the Liapunov function $V(x(t))=\|x(t)\|^2_2$ is decreasing on trajectories. 

The question naturally arises: if the solution $x$ of \eqref{eq.unperturbed}
obeys \eqref{eq.xglobalstable}, under what conditions on $f$ and
$\sigma$ does the solution $X$ of \eqref{eq.intromainstocheqn} obey
\begin{equation} \label{eq.stochglobalstable}
\lim_{t\to\infty} X(t,\xi)=0, \quad \text{a.s. for each
$\xi\in\mathbb{R}^d$}.
\end{equation}
The convergence phenomenon captured in \eqref{eq.stochglobalstable}
for the solution of \eqref{eq.intromainstocheqn} is often called
almost sure \emph{global convergence} (or \emph{global stability}
for the solution of \eqref{eq.unperturbed}), because the solution of
the perturbed equation \eqref{eq.intromainstocheqn} converges to the
zero equilibrium solution of the underlying unperturbed equation
\eqref{eq.unperturbed}.

In the case when $d=1$ i.e., for the scalar equation, a series of
papers has progressively lead to a characterisation of the almost
sure convergence embodied in \eqref{eq.stochglobalstable}. It was
shown in Chan and Williams~\cite{ChanWill:1989} that if $f$ is
strictly increasing with $f(0)=0$ and $f$ obeys
\begin{equation} \label{eq.ftoinfty}
\lim_{x\to\infty} f(x)=\infty, \quad \lim_{x\to-\infty}
f(x)=-\infty,
\end{equation}
then the solution $X$ of \eqref{eq.intromainstocheqn} obeys
\eqref{eq.stochglobalstable} holds if $\sigma$ obeys
\begin{equation} \label{eq.sigmalogto0}
\lim_{t\to\infty} \sigma^2(t)\log t = 0.
\end{equation}
Moreover, Chan and Williams also proved, if $t\mapsto\sigma^2(t)$ is
decreasing to zero, that if the solution $X$ of
\eqref{eq.intromainstocheqn} obeys \eqref{eq.stochglobalstable},
then $\sigma$ must obey \eqref{eq.sigmalogto0}. These results were
extended to finite--dimensions by Chan in~\cite{Chan:1989}. The
results in~\cite{ChanWill:1989,Chan:1989} are motivated by problems
in simulated annealing.

In Appleby, Gleeson and Rodkina~\cite{JAJGAR:2009}, the monotonicity
condition on $f$ and \eqref{eq.ftoinfty} were relaxed. It was shown
if $f$ obeys \eqref{eq.floclip} and
\eqref{eq.introfglobalunperturbed}, and in place of
\eqref{eq.ftoinfty} also obeys
\begin{equation} \label{eq.fliminfinfty}
\text{There exists $\phi>0$ such that }\phi:=\liminf_{|x|\to\infty}
|f(x)|,
\end{equation}
then the solution $X$ of \eqref{eq.mainstocheqn} obeys
\eqref{eq.stochglobalstable} holds if $\sigma$ obeys
\eqref{eq.sigmalogto0}. The converse of Chan and Williams is also
established: if $t\mapsto\sigma^2(t)$ is decreasing, and the
solution $X$ of \eqref{eq.mainstocheqn} obeys
\eqref{eq.stochglobalstable}, then $\sigma$ must obey
\eqref{eq.sigmalogto0}. Moreover, it was also shown, without
monotonicity on $\sigma$, that if
\begin{equation} \label{eq.sigmalogtoinfty}
\lim_{t\to\infty} \sigma^2(t)\log t = +\infty,
\end{equation}
then the solution $X$ of \eqref{eq.mainstocheqn} obeys
\begin{equation} \label{eq.stochunstable}
\limsup_{t\to\infty} |X(t,\xi)|=+\infty, \quad \text{a.s. for each
$\xi\in\mathbb{R}$}.
\end{equation}
Furthermore, it was shown that the condition \eqref{eq.sigmalogto0}
could be replaced by the weaker condition
\begin{equation} \label{eq.capSigma0}
\lim_{t\to\infty} \int_0^t e^{-2(t-s)}\sigma^2(s)\,ds \cdot \log\log
\int_0^t \sigma^2(s)e^{2s}\,ds = 0
\end{equation}
and that \eqref{eq.capSigma0} and \eqref{eq.sigmalogto0} are
equivalent when $t\mapsto \sigma^2(t)$ is decreasing. In fact, it
was even shown that if $\sigma^2$ is not monotone decreasing,
$\sigma$ does not have to satisfy \eqref{eq.sigmalogto0} in order
for $X$ to obey \eqref{eq.stochglobalstable}.

Finally, in \cite{JAAR:2010a}, it was shown under the scalar version
of condition \eqref{eq.introfglobalunperturbed} that the solution
$X$ of \eqref{eq.mainstocheqn} obeys \eqref{eq.stochglobalstable} if
and only if $\sigma$ obeys
\begin{equation} \label{eq.sigmaiffXto0scalar}
S_h(\epsilon):= \sum_{n=1}^\infty \sqrt{\int_{nh}^{(n+1)h}
\sigma^2(s)\,ds} \cdot
\exp\left(-\frac{\epsilon^2}{2\int_{nh}^{(n+1)h}
\sigma^2(s)\,ds}\right)
\end{equation}
for every $\epsilon>0$. It can therefore be seen that this result
does not require monotonicity conditions on $\sigma$ or on $f$ in
order to characterise the global convergence of solutions of
\eqref{eq.intromainstocheqn}, nor asymptotic information on $f$ such
as \eqref{eq.fliminfinfty}. Moreover it is shown that if
\eqref{eq.sigmaiffXto0scalar} does not hold, then
$\mathbb{P}[X(t)\to 0 \text{ as $t\to\infty$}]=0$ for any
$\xi\in\mathbb{R}$.


In this paper, we extend the results of \cite{JAAR:2010a} to finite
dimensions. Our first main result (Theorem~\ref{theorem.Xiffsigma}) shows
that if $f$ obeys \eqref{eq.introfglobalunperturbed} and is continuous, and $\sigma$ is also continuous, then any
solution $X$ of \eqref{eq.intromainstocheqn} obeys
\eqref{eq.stochglobalstable} if and only if
\begin{multline} \label{eq.sigmaiffXto0}
S_h'(\epsilon)=\sum_{n=0}^\infty
 \sqrt{\int_{nh}^{(n+1)h} \|\sigma(s)\|^2_F\,ds} \cdot
\exp\left(-\frac{\epsilon^2}{2\int_{nh}^{(n+1)h} \|\sigma(s)\|^2_F\,ds}\right)<+\infty,\quad
\\\text{for every $\epsilon>0$},
\end{multline}
provided that $f$ obeys
\begin{equation} \label{eq.introfasy}
\text{There exists $\phi>0$ such that }  \phi:=\liminf_{x\to\infty}
\inf_{\|y\|=x}\langle y,f(y)\rangle,
\end{equation}
a condition weaker than, but similar to, \eqref{eq.fliminfinfty}. We note that in the scalar case the assumption \eqref{eq.introfasy} is not needed
in order to characterise global stability; all that is required is the scalar analogue of \eqref{eq.sigmaiffXto0}. It is also notable that the assumption of
Lipschitz continuity can be dispensed with, the potential cost being that there may be more than one solution of the differential equation \eqref{eq.intromainstocheqn}. Of course, if $f$ is additionally assumed to be locally Lipschitz continuous, or obey a one--sided Lipschitz continuity condition, then there is a unique continuous adapted process obeying \eqref{eq.intromainstocheqn}.

In the case when \eqref{eq.introfasy} is not assumed, it can still be
shown that if \eqref{eq.sigmaiffXto0} does not hold, then
\[
\mathbb{P}[X(t,\xi)\to 0 \text{ as $t\to\infty$}]=0 \text{ for each
$\xi\in\mathbb{R}^d$}.
\]
Also, if \eqref{eq.sigmaiffXto0} holds, the only possible limiting
behaviour of solutions are that $X(t)\to 0$ as $t\to\infty$ or
$\|X(t)\|\to\infty$ as $t\to\infty$ (Theorem~\ref{theorem.Xdatanof}). If the noise intensity is sufficiently small, in the sense that $\sigma\in
L^2(0,\infty)$, it can be shown that $X$ obeys \eqref{eq.stochglobalstable} without any
further conditions on $f$. In the case when the sum in \eqref{eq.sigmaiffXto0} is infinite for all $\epsilon>0$, it can  be
shown \emph{a fortiori} that $\limsup_{t\to\infty} \|X(t)\|=+\infty$ a.s., while if $S_h'(\epsilon)$ is finite for some $\epsilon$ but infinite for others, it can be shown that  $\limsup_{t\to\infty} \|X(t)\|$ is bounded below by a constant a.s. These results are the subject of Theorem~\ref{theorem.Xlowerboundunbound}.

The other major result in the paper (Theorem~\ref{theorem.Xclassify}) gives a complete classification of
the asymptotic behaviour of solutions of \eqref{eq.mainstocheqn}
under a strengthening of \eqref{eq.introfasy}, namely
\begin{equation}\label{eq.introfasybounded}
\liminf_{r\to\infty} \inf_{\|x\|=r}\frac{\langle
x,f(x)\rangle}{\|x\|}=+\infty,
\end{equation}
which is a direct analogue of the condition needed
to give a classification of solutions of \eqref{eq.mainstocheqn} in
the scalar case. We show that solutions of \eqref{eq.mainstocheqn}
are either (a) convergent to zero with probability one (b) bounded, not convergent to zero,
but approach zero arbitrarily close infinitely often with probability one or (c) are unbounded
with probability one. Possibility (a) occurs when $S_h'(\epsilon)$ is finite for all $\epsilon$; (b)
happens when $S_h'(\epsilon)$ is finite for some $\epsilon$, but infinite for others, and (c) occurs
when $S_h'(\epsilon)$ is infinite for all $\epsilon$. Once again, we do not need the assumption of Lipschitz continuity.

It was shown in \cite{JAJCAR:2012} that these conditions characterised the stability, boundedness and unboundedness of solutions of
affine stochastic differential equations with the same state--independent diffusion coefficient, contingent on the deterministic part
of the equation yielding globally stable solutions. Therefore, we see that the long--run behaviour demonstrates relatively little
sensitivity to the type of nonlinearity present in the dirft term. In fact, this lack of sensitivity is even more pronounced when one considers stability within the class of SDEs with dissipative drift condition, because as the same asymptotic behaviour results irrespective of the strength of the nonlinearity $f$, provided that $f$ is of order $1/\|x\|$ or greater
as $\|x\|\to\infty$, as characterised by \eqref{eq.introfasy}.

Although \eqref{eq.sigmaiffXto0} is necessary and sufficient for $X$ to obey $\lim_{t\to\infty} X(t)=0$
a.s., these conditions may be hard to apply in practice. For this
reason we also deduce sharp sufficient conditions on $\sigma$ which
enable us to determine for which value of $\epsilon$ the function
$S_h'(\epsilon)$ is finite. One such condition is
the following: if it is known for some $c>0$ that
\[
\lim_{t\to\infty} \int_t^{t+c} \|\sigma(s)\|^2_F\,ds \log t = L\in
[0,\infty],
\]
then $L=0$ implies that $X$ tends to zero a.s.; $L$ being positive
and finite implies $X$ is bounded, but does not converge to zero;
and $L$ being infinite implies $X$ is unbounded. This result is stated as Theorem~\ref{theorem.XsiglogtL}.
In the case when
$t\mapsto \|\sigma(t)\|^2=:\Sigma_1(t)^2$ or $t\mapsto\int_t^{t+1}
\|\sigma(s)\|^2\,ds=:\Sigma_2(t)^2$ are nonincreasing functions, it
can also be seen that $X(t)\to 0$ as $t\to\infty$ a.s. is equivalent
to $\lim_{t\to\infty}\Sigma_i(t)^2\log t=0$; this is the subject of Theorem~\ref{theorem.eqvt}.

The main results are proven by showing that the stability of
\eqref{eq.intromainstocheqn} is intimately connected with the the
stability of a linear SDE with the same diffusion coefficient
(Theorem~\ref{th.xto0yto0}). The asymptotic behaviour of the linear SDE
has been characterised in \cite{JAJCAR:2012}, and the relevant results are restated here for
the reader's convenience.
As to the organisation of the paper, notation, and
statements and discussion about main results are presented in
Section 2, with the proofs of these results being in the main part
deferred to Section 3. The proof concerning upper bounds on the solution turns out to present 
the most challenges, and accordingly the enrirety of Section 4 is devoted to its proof.

\section{Statement and Discussion of Main Results}
\subsection{Notation}
In advance of stating and discussing our main results, we introduce
some standard notation. Let $d$ and $r$ be integers. We denote by
$\mathbb{R}^d$ $d$--dimensional real--space, and by
$\mathbb{R}^{d\times r}$ the space of $d\times r$ matrices with real
entries. Here $\mathbb{R}$ denotes the set of real numbers. We
denote the maximum of the real numbers $x$ and $y$ by $x\vee y$ and
the minimum of $x$ and $y$ by $x\wedge y$. If $x$ and $y$ are in
$\mathbb{R}^d$, the standard innerproduct of $x$ and $y$ is denoted
by $\langle x,y\rangle$. The standard Euclidean norm on
$\mathbb{R}^d$ induced by this innerproduct is denoted by $\|\cdot\|$.
If $A\in \mathbb{R}^{d\times r}$, we denote the entry in the $i$--th
row and $j$--th column by $A_{ij}$. For $A\in \mathbb{R}^{d\times
r}$ we denote the Frobenius norm of $A$ by
\[
\|A\|_F=\left( \sum_{j=1}^r \sum_{i=1}^d \|A_{ij}\|^2\right)^{1/2}.
\]
Let $C(I;J)$ denote the space of continuous functions $f:I\to J$
where $I$ is an interval contained in $\mathbb{R}$ and $J$ is a
finite dimensional Banach space. We denote by
$L^2([0,\infty);\mathbb{R}^{d\times r})$ the space of Lebesgue
square integrable functions $f:[0,\infty)\to\mathbb{R}^{d\times r}$
such that $\int_0^\infty \|f(s)\|_F^2\,ds < + \infty$.

\subsection{Set--up of the problem}
Let $d$ and $r$ be integers. We fix a complete filtered probability
space $(\Omega,\mathcal{F}, (\mathcal{F}(t))_{t\geq 0},\mathbb{P})$.
Let $B$ be a standard $r$--dimensional Brownian motion which is
adapted to $(\mathcal{F}(t))_{t\geq 0}$. We consider the stochastic
differential equation
\begin{equation} \label{eq.mainstocheqn}
dX(t)=-f(X(t))\,dt + \sigma(t)\,dB(t), \quad t\geq 0; \quad
X(0)=\xi\in\mathbb{R}^d.
\end{equation}
We suppose that
\begin{equation} \label{eq.fglobalunperturbed}
f\in C(\mathbb{R}^d;\mathbb{R}^d); \quad \langle x,f(x)\rangle >0,
\quad x\neq 0; \quad f(0)=0,
\end{equation}
and that $\sigma$ obeys
\begin{equation} \label{eq.sigmacns}
\sigma \in C([0,\infty);\mathbb{R}^{d\times r}).
\end{equation}
To simplify the existence and uniqueness of a unique continuous
adapted solution of \eqref{eq.mainstocheqn} on $[0,\infty)$, we
may assume that
\begin{equation} \label{eq.floclip}
\text{$f$ is locally Lipschitz continuous}.
\end{equation}
See e.g., \cite{Mao1}. This ensures the existence of a unique solution up to
a (possibly infinite) explosion time. In the case that there is a unique continuous adapted process obeying \eqref{eq.mainstocheqn}, we refer to it as the (local) solution of
\eqref{eq.mainstocheqn}. Another Lipschitz--like condition on $f$ which guarantees the uniqueness of solutions is
that there exists $K\geq 0$ such that
\begin{equation} \label{eq.onesidedlip}
-\langle x-y,f(x)-f(y)\rangle \leq K\|x-y\|^2_2 \quad\text{for all $x,y\in \mathbb{R}^d$}.
\end{equation}
This is often referred to as a one--sided Lipschitz condition. This condition is not inconsistent with
\eqref{eq.fglobalunperturbed}: notice that putting $y=0$ in \eqref{eq.onesidedlip} yields $\langle x,f(x)\rangle \geq -K\|x\|^2_2$ for all $x\in\mathbb{R}^d$, which is true by \eqref{eq.fglobalunperturbed}.
In the case when the equation is scalar, and $K=0$, then \eqref{eq.onesidedlip} is nothing other than the monotonicity
of $f$, a hypothesis favoured by Chan and Williams in their asymptotic analysis.

Therefore, it remains to answer the question as to whether \eqref{eq.floclip} can be relaxed and still ensure the existence of a local solution, and also whether the local solution is global. Granted that $f$ is continuous, the answer to the question of the existence of a local solution is positive. Regardless of whether local solutions are unique, it is standard to show that any local solution exists on $[0,\infty)$ a.s. This is guaranteed by the dissipative condition in \eqref{eq.fglobalunperturbed}. Therefore, any local solution is global. These claims are justified in the next result.
\begin{proposition} \label{prop.exist}
Suppose that $f$ obeys \eqref{eq.fglobalunperturbed} and that $\sigma$ obeys \eqref{eq.sigmacns}. Then there exists
a continuous adapted process that obeys \eqref{eq.mainstocheqn} for all $t\geq 0$ a.s.
\end{proposition}
The proof is quite routine, and we make no claim that this represents an advance in substance or in sophistication on extant results in the 
existence theory of stochastic differential equations. However, we find it convenient to fashion an existence result that makes use of the types of 
hypotheses on $f$ and $\sigma$ that are of significance when making a study of the asymptotic behaviour of \eqref{eq.mainstocheqn}, and these considerations 
lead us to include the result and its proof here.  

\subsection{Asymptotic classification of an affine equation}
In this section, we state some results proven in Appleby, Cheng and Rodkina~\cite{JAJCAR:2012} which concern
the classification of affine stochastic differential equations (i.e., equations in which $f$ is a linear function). It transpires
that it is enough for the purposes of the current work to understand the behaviour for a single affine stochastic differential equation. The crucial property of this equation is that it has the same diffusion coefficient as
the solution $X$ of \eqref{eq.mainstocheqn} to tend to zero. The desired process $Y$ is defined to be the unique continuous
adapted process which obeys the stochastic differential equation
\begin{equation} \label{def.Y}
dY(t)=-Y(t)\,dt + \sigma(t)\,dB(t), \quad t\geq 0; \quad Y(0)=0.
\end{equation}
Note that $Y$ has the representation
\begin{equation}   \label{eq.Yform}
Y(t)=e^{-t}\int_0^t e^s \sigma(s)\,dB(s), \quad t\geq 0.
\end{equation}
%
Define
\begin{equation} \label{def.Shpr}
S_h'(\epsilon)=\sum_{n=1}^\infty \sqrt{\int_{nh}^{(n+1)h} \|\sigma(s)\|^2_F\,ds} \cdot
\exp\left(-\frac{\epsilon^2}{2\int_{nh}^{(n+1)h} \|\sigma(s)\|^2_F\,ds}\right),
\end{equation}
Since $S_h'$ is a monotone function of $\epsilon$, it is the case
that either (i) $S_h'(\epsilon)$ is finite for all $\epsilon>0$; (ii)
there is $\epsilon'>0$ such that for all $\epsilon>\epsilon'$ we
have $S_h'(\epsilon)<+\infty$ and $S_h'(\epsilon)=+\infty$ for all
$\epsilon<\epsilon'$; and (iii) $S_h'(\epsilon)=+\infty$ for all
$\epsilon>0$.

%

Armed with these observations, we see that the following theorem, which appears in \cite{JAJCAR:2012}
characterises the pathwise asymptotic behaviour of solutions of \eqref{def.Y}.  In the scalar
case it yields a result of Appleby, Cheng and Rodkina in~\cite{JAJCAR:2011dresden} when $h=1$. It is also of utility
when considering the relationship between the asymptotic behaviour
of solutions of stochastic differential equations and the asymptotic
behaviour of uniform step--size discretisations.
\begin{theorem}  \label{theorem.Yclassify}
Suppose that $\sigma$ obeys \eqref{eq.sigmacns} and $Y$ is the
unique continuous adapted process which obeys \eqref{def.Y}.
Suppose that $S_h'$ is defined by \eqref{def.Shpr}.
\begin{itemize}
\item[(A)] If
\begin{equation} \label{eq.thetastableh}
\text{$S_h'(\epsilon)$ is finite for all $\epsilon>0$},
\end{equation}
then
\begin{equation} \label{eq.Ytto0}
\lim_{t\to\infty} Y(t)=0, \quad\text{a.s.}
\end{equation}
\item[(B)] If there exists $\epsilon'>0$ such that
\begin{equation}  \label{eq.thetaboundedh}
\text{$S_h'(\epsilon)$ is finite for all $\epsilon>\epsilon'$}, \quad
\text{$S_h'(\epsilon)=+\infty$ for all $\epsilon<\epsilon'$},
\end{equation}
then there exists deterministic $0<c_1\leq c_2<+\infty$ such that
\begin{equation} \label{eq.Ytboundedh}
c_1\leq \limsup_{t\to\infty} \|Y(t)\|\leq c_2, \quad \text{a.s.}
\end{equation}
Moreover
\begin{equation} \label{eq.liminfYaveY0}
\liminf_{t\to\infty} \|Y(t)\|=0, \quad\lim_{t\to\infty}
\frac{1}{t}\int_0^t \|Y(s)\|^2\,ds=0, \quad\text{a.s.}
\end{equation}
\item[(C)] If
\begin{equation} \label{eq.thetaunstableh}
\text{$S_h'(\epsilon)=+\infty$ for all $\epsilon>0$},
\end{equation}
then
\begin{equation} \label{eq.Ytunstable}
\limsup_{t\to\infty} \|Y(t)\|=+\infty, \quad \text{a.s.}
\end{equation}
\end{itemize}
\end{theorem}
The conditions and form of Theorem~\ref{theorem.Yclassify}, as well
as other theorems in this section, are inspired by those of
\cite[Theorem 1]{ChanWill:1989} and by \cite[Theorem 6, Corollary
7]{JAARMR:2009}.

Another result from \cite{JAJCAR:2012} that of is utility is that the parameter $h>0$ in Theorem~\ref{theorem.Yclassify}, while potentially of interest for numerical simulations, plays no role in classifying the dynamics of \eqref{def.Y}.
Therefore, we may take $h=1$ without loss of generality.
\begin{proposition} \label{prop.ShS1}
Suppose that $S_h'$ is defined by \eqref{def.Shpr}.
\begin{itemize}
\item[(i)] If  $S_1'(\epsilon)<+\infty$ for all $\epsilon>0$, then for each $h>0$ we have $S_h'(\epsilon)<+\infty$ for all $\epsilon>0$.
\item[(ii)] If there exists $\epsilon'>0$ such that $S_1'(\epsilon)<+\infty$ for all $\epsilon>\epsilon'$ and
$S_1'(\epsilon)=+\infty$ for all $\epsilon<\epsilon'$, then for each $h>0$ there exists $\epsilon_h'>0$ such that
 $S_h'(\epsilon)<+\infty$ for all $\epsilon>\epsilon_h'$ and  $S_h'(\epsilon)=+\infty$ for all $\epsilon<\epsilon_h'$.
\item[(iii)] If  $S_1'(\epsilon)=+\infty$ for all $\epsilon>0$, then for each $h>0$ we have $S_h'(\epsilon)=+\infty$ for all $\epsilon>0$.
\end{itemize}
\end{proposition}

Given that the equations studied are in continuous time, it is natural to ask whether
the summation conditions can be replaced by integral conditions on $\sigma$ instead. The answer is in
the affirmative.
To this end we introduce for fixed $c>0$ the $\epsilon$--dependent
integral
\begin{equation} \label{def.Ieps}
I_c(\epsilon)=\int_0^\infty \varsigma_c(t)
\exp\left(-\frac{\epsilon^2/2}{\varsigma_c(t)^2
}\right)\chi_{(0,\infty)}\left( \varsigma_c(t)\right)\,dt,
\end{equation}
where we have defined
\begin{equation} \label{def.varsigc}
\varsigma_c(t):= \left(\int_{t}^{t+c}
\|\sigma(s)\|^2_F\,ds\right)^{1/2}, \quad t\geq 0.
\end{equation}
We notice that $\epsilon\mapsto I_c(\epsilon)$ is a monotone function,
and therefore $I_c(\cdot)$ is either finite for all $\epsilon>0$;
infinite for all $\epsilon>0$; or finite for all
$\epsilon>\epsilon'$ and infinite for all $\epsilon<\epsilon'$. The
following theorem is therefore seen to classify the asymptotic
behaviour of \eqref{def.Y}.
\begin{theorem} \label{theorem.Icondn}
Suppose that $\sigma$ obeys \eqref{eq.sigmacns} and that $Y$ is the
unique continuous adapted process which obeys \eqref{def.Y}. Let $c>0$,
$I_c(\cdot)$ be defined by \eqref{def.Ieps}, and $\varsigma_c$ by \eqref{def.varsigc}.
\begin{itemize}
\item[(A)] If
\begin{equation} \label{eq.Istable}
\text{$I_c(\epsilon)$ is finite for all $\epsilon>0$},
\end{equation}
then $\lim_{t\to\infty} Y(t)=0$ a.s.
\item[(B)] If there exists $\epsilon'>0$ such that
\begin{equation}  \label{eq.Ibounded}
\text{$I_c(\epsilon)$ is finite for all $\epsilon>\epsilon'$}, \quad
\text{$I_c(\epsilon)=+\infty$ for all $\epsilon<\epsilon'$},
\end{equation}
then there exist deterministic $0<c_1\leq c_2<+\infty$ such that
\[
c_1\leq \limsup_{t\to\infty} \|Y(t)\|\leq c_2, \quad\text{a.s.}
\]
Moreover, $Y$ also obeys  \eqref{eq.liminfYaveY0}.
\item[(C)] If
\begin{equation} \label{eq.Iunstable}
\text{$I_c(\epsilon)=+\infty$ for all $\epsilon>0$},
\end{equation}
then $\limsup_{t\to\infty} \|Y(t)\|=+\infty$ a.s.
\end{itemize}
\end{theorem}
A consequence of this result and of Theorem~\ref{theorem.Yclassify}
is that $S_h'(\epsilon)<+\infty$ for all $\epsilon>0$ if and only if $I_c(\epsilon)<+\infty$
for all $\epsilon>0$; that  $S_h'(\epsilon)=+\infty$ for all $\epsilon>0$ if and only if $I_c(\epsilon)=+\infty$
for all $\epsilon>0$; and that there exists $\epsilon'>0$ such that  $S_h'(\epsilon)<+\infty$ for all $\epsilon>\epsilon'$
and  $S_h'(\epsilon)=+\infty$ for all $\epsilon<\epsilon'$ if and only if there exists $\epsilon^\ast>0$ such that
$I_c(\epsilon)<+\infty$ for all $\epsilon>\epsilon^\ast$ and $I_c(\epsilon)=+\infty$ for all $\epsilon<\epsilon^\ast$.
Therefore, in all the results in the next section, we may replace, if we prefer, any condition relating to $S_h'$ with a condition involving 
the integral $I_c$. By norm equivalence, it is also the case that the Frobenius norm of $\sigma$ can be replaced by any other norm on 
$\mathbb{R}^{d\times r}$, and that the finiteness properties of $I_c$ and $S_h'$ are preserved for any other norm.

\subsection{Statement and discussion of main results} 
We now turn our attention to the nonlinear equation
\eqref{eq.mainstocheqn}. We start by showing that solutions will
become arbitrarily large whenever the diffusion coefficient is such
that solutions of the corresponding affine equation \eqref{def.Y}
have the same property. Furthermore, if solutions are of
\eqref{def.Y} are bounded but not convergent to zero, then solutions
of \eqref{eq.mainstocheqn} do not converge to zero.
\begin{theorem}\label{theorem.Xlowerboundunbound}
Suppose that $f$ is continuous. Suppose that $\sigma$
obeys \eqref{eq.sigmacns} and let $S_h'$ be defined by \eqref{def.Shpr}.
Suppose that $X$ is a continuous adapted process which obeys \eqref{eq.mainstocheqn}.
\begin{itemize}
\item[(A)] Suppose that $S_h'$ obeys \eqref{eq.thetaunstableh}. Then
\[
\limsup_{t\to\infty} \|X(t)\|=+\infty, \quad \text{a.s.}
\]
\item[(B)] Suppose that $S_h'$ obeys \eqref{eq.thetaboundedh}. Then
there is a deterministic $c_3>0$ such that
\[
\limsup_{t\to\infty} \|X(t)\|\geq c_3, \quad\text{a.s.}
\]
\end{itemize}
\end{theorem}
 We note in this result, as well as in the rest of the results in this paper, that we do not require $f$ to obey the local Lipschitz condition.
The price to be paid for this is that the solution of the equation need not be unique. If uniqueness is desired, the local Lipschitz condition, or one--sided global Lipschitz condition can be imposed. However, it is interesting to note that should solutions exist,
they must all share the same asymptotic behaviour.

We show that its solutions can either tend to zero or their modulus
tends to infinity if and only if solutions of a linear equation with
the same diffusion tend to zero.
\begin{theorem} \label{th.xto0yto0}
Suppose that $f$ satisfies \eqref{eq.fglobalunperturbed}.
Suppose that $\sigma$ obeys \eqref{eq.sigmacns}.
Suppose that $X$ is a continuous adapted process which obeys \eqref{eq.mainstocheqn}.
Let $Y$ be the unique continuous adapted process which obeys of \eqref{def.Y}.
Then there exist a.s. events $\Omega_1$
and $\Omega_2$ such that
\begin{gather}\label{eq.Y0impX0}
\{\omega:\lim_{t\to\infty} X(t,\omega)=0\} \subseteq
\{\omega: \lim_{t\to\infty}Y(t,\omega)=0\}\cap \Omega_1,\\
\label{eq.X0impY0} \{\omega:\lim_{t\to\infty} Y(t,\omega)=0\}
\subseteq \{\omega: \lim_{t\to\infty} X(t,\omega)=0\}\cup
\{\omega:\lim_{t\to\infty} \|X(t,\omega)\|=\infty\} \cap \Omega_2.
\end{gather}
\end{theorem}
When taken in conjunction with Theorem~\ref{theorem.Yclassify}, we
see that the condition \eqref{eq.thetastableh} comes close to
characterising the convergence of solutions of
\eqref{eq.mainstocheqn} to zero, contingent on the possibility that
$\|X(t)\|\to\infty$ as $t\to\infty$ being eliminated.
\begin{theorem} \label{theorem.Xdatanof}
Suppose that $f$ satisfies \eqref{eq.fglobalunperturbed}.
Suppose that $\sigma$ obeys \eqref{eq.sigmacns}.
Let $X$ be a continuous adapted process which obeys \eqref{eq.mainstocheqn}.
\begin{itemize}
\item[(i)] If $\sigma$ obeys \eqref{eq.thetastableh}, then for each $\xi\in\mathbb{R}^d$,
\[
\{\lim_{t\to\infty} \|X(t,\xi)\|=\infty\}\cup \{\lim_{t\to\infty}
\|X(t,\xi)\|=0\} \quad\text{is an a.s. event}.
\]
\item[(ii)] If $X(t,\xi)\to 0$ with positive probability for some $\xi\in \mathbb{R}^d$, then $\sigma$ obeys \eqref{eq.thetastableh}.
\end{itemize}
\end{theorem}
\begin{proof}
To prove part (i), we first note that \eqref{eq.thetastableh} and
Theorem~\ref{theorem.Yclassify} implies that $Y(t)\to 0$ as
$t\to\infty$ a.s. Theorem~\ref{th.xto0yto0} then implies that the
event $\{\lim_{t\to\infty} \|X(t,\xi)\|=\infty\}\cup
\{\lim_{t\to\infty} X(t,\xi)=0\}$ is a.s. To show part (ii), by
hypothesis and Theorem~\ref{th.xto0yto0}, we see that
$\mathbb{P}[Y(t)\to 0 \text{ as $t\to\infty$}]>0$. Therefore, by
Theorem~\ref{theorem.Yclassify}, it follows that $\sigma$ obeys
\eqref{eq.thetastableh}.
\end{proof}

Part (i) of Theorem~\ref{theorem.Xdatanof} is unsatisfactory, as it
does not rule out the possibility that $\|X(t)\|\to \infty$ as
$t\to\infty$ with positive probability. If further restrictions are
imposed on $f$ and $\sigma$, however, it is possible to conclude
that $X(t,\xi)\to 0$ as $t\to\infty$ a.s. In the scalar case, it was
shown in Appleby, Cheng and Rodkina~\cite{JAAR:2010a} that no such
additional conditions are required.

Our first result in this direction imposes an extra condition on
$\sigma$, but not on $f$. We note that when $\sigma\in
L^2([0,\infty);\mathbb{R}^{d\times r})$, $Y$ obeys \eqref{eq.Ytto0}
and that $X$ obeys \eqref{eq.stochglobalstable}.
To prove the result, we apply a semimartingale convergence theorem of
Lipster--Shiryaev (see e.g., \cite[Theorem 7,
p.139]{LipShir:1989} or~\cite[Theorem 3.9]{Mao1}) to the non--negative semimartingale $\|X\|^2$.
We state the desired semimartingale convergence result for the reader's ease of reference.
\begin{lemma}\label{theorem.convergence}
Let $\left\{A(t)\right\}_{t\geq 0}$ and $\left\{U(t)\right\}_{t\geq 0}$ be two continuous adapted increasing process with
$A(0)=U(0)=0$ a.s. Let $\left\{M(t)\right\}_{t\geq 0}$ be a real--valued continuous local martingale with $M(0)=0$ a.s.
Let $\xi$ be a nonnegative $\mathcal{F}_0$--measurable random variable. Define
\[
Z(t)=\xi+A(t)-U(t)+M(t)\quad for\quad t\geq 0.
\]
If $Z(t)$ is nonnegative, then
\[
\left\{\lim_{t\to \infty}A(t)<\infty\right\}\subset\left\{\lim_{t\to \infty}Z(t) \text{exists and is finite}\right\}
\cap\left\{\lim_{t\to \infty}U(t)<\infty\right\} a.s.
\]
where $B\subset D$ a.s. means $\mathbb{P}(B\cap D^c)=0$. In particular, if $\lim_{t\to\infty}A(t)<\infty$ a.s., then for almost all $\omega \in \Omega$
\[
\lim_{t\to\infty}Z(t,\omega) \text{exists and is finite, and} \lim_{t\to\infty}U(t,\omega)<\infty.
\]
\end{lemma}
Applying Lemma~\ref{theorem.convergence}, we can establish the following result.
\begin{theorem} \label{th.sigL2}
Suppose that $f$ satisfies \eqref{eq.fglobalunperturbed}.
Suppose that $\sigma$ obeys \eqref{eq.sigmacns}
and $\sigma\in L^2([0,\infty);\mathbb{R}^{d\times r})$. Suppose that $X$ is a continuous
adapted process which obeys \eqref{eq.mainstocheqn}, and let $Y$ be the unique continuous adapted process
that obeys \eqref{def.Y}. Then $X$ obeys \eqref{eq.stochglobalstable} and
$\lim_{t\to\infty} Y(t)=0$ a.s.
\end{theorem}
It can be seen from Theorem~\ref{th.sigL2} that it only remains to
prove Theorem~\ref{th.xto0yto0} in the case when $\sigma\not\in
L^2([0,\infty);\mathbb{R}^{d\times r})$. Under an additional
restriction on $f$ (but no extra condition on $\sigma$) we can give
necessary and sufficient conditions in terms of $\sigma$ for which
$X$ tends to zero a.s.
\begin{theorem} \label{theorem.Xiffsigma}
Suppose $f$ obeys \eqref{eq.fglobalunperturbed} and
\begin{equation} \label{eq.fasy}
\liminf_{r\to\infty} \inf_{\|x\|=r} \langle x,f(x)\rangle>0.
\end{equation}
Suppose that $\sigma$ obeys \eqref{eq.sigmacns}.
Suppose that $X$ is a continuous adapted process which obeys \eqref{eq.mainstocheqn}.
Then
the following are equivalent:
\begin{itemize}
\item[(A)] $S_h'$ obeys \eqref{eq.thetastableh};
\item[(B)] $\lim_{t\to\infty} X(t,\xi)=0$ with positive probability for some $\xi\in\mathbb{R}^d$.
\item[(C)] $\lim_{t\to\infty} X(t,\xi)=0$ a.s. for each $\xi\in\mathbb{R}^d$.
\end{itemize}
\end{theorem}
Notice that no monotonicity conditions are required on
$\|\sigma\|^2_F$ in order for this result to hold. The condition
\eqref{eq.fasy} was not required to prove an analogous result in the
scalar case in~\cite{JAAR:2010a}. However, the condition is weaker
than the condition \eqref{eq.fliminfinfty} which was required in the
scalar case to secure the stability of solutions of
\eqref{eq.mainstocheqn} in \cite{JAJGAR:2009}.

We pause temporarily to discuss the condition \eqref{eq.fasy}. Is it a purely technical condition, which makes the proof of Theorem~\ref{theorem.Xiffsigma} more convenient, or is it representative of a class of conditions whose role is to provide some minimal strength of asymptotic equilibrium reversion in the finite--dimensional case,
so that stability is preserved when stochastic perturbations are present? We speculate that the condition is of the latter type. This is because the stochastic part of the equation can be \emph{transient} (in the sense that its norm can grow to infinity as $t\to\infty$).
An example of this possibility was given in  \cite{JAJCAR:2012}. 
In the scalar case we do not need any additional condition on $f$ because the perturbation $\int_0^t \sigma(s)\,dB(s)$,
being a time--changed one--dimensional Brownian motion, is \emph{recurrent}. 

To give some motivation as to why we expect some extra condition on $f$ in the presence of a cumulatively transient perturbation, we recall the deterministic results in Appleby and Cheng~\cite{JAJC:2011szeged}, and write the differential equation
\begin{equation*}  
x'(t)=-f(x(t))+g(t), \quad t>0; \quad x(0)=\xi,
\end{equation*}
in the integral form
\begin{equation} \label{eq.ch4odeintro}
x(t)=\xi-\int_0^t f(x(s))\,ds + \int_0^t g(s)\,ds, \quad t\geq 0.
\end{equation}
In the case when $g(t)\to 0$ but $\int_0^t g(s)\,ds=+\infty$ as $t\to\infty$, we have shown that unless $f$ has enough
strength to counteract the cumulative perturbation $\int_0^t g(s)\,ds$, it is possible that $x(t)\to\infty$ as $t\to\infty$.
If one writes the stochastic equation in integral form
\[
X(t)=\xi-\int_0^t f(X(s))\,ds + \int_0^t \sigma(s)\,dB(s), \quad t\geq 0,
\]
we can guess that when the cumulative perturbation $\int_0^t \sigma(s)\,dB(s)$ is not convergent (which
happens when $\sigma\notin L^2([0,\infty);\mathbb{R}^{d\times r})$), some minimal strength in $f$ may be needed to
keep the solution from escaping to infinity.

Of course, it is speculative to suggest that one can make inferences about the asymptotic behaviour of stochastic equations from deterministic ones, however 
close the structural correspondence. But for this class of equations, we already have evidence that there are remarkably close connections between 
admissible types of perturbations, and this also tend to justify the analogy between stochastic and deterministic equations. 
In the case when $g$ is in $L^1(0,\infty)$ and the cumulative perturbation $\int_0^t g(s)\,ds$ converges, it is
well--known (cf. e.g.,~\cite{JAJC:2011szeged}) that the solution of \eqref{eq.ch4odeintro} obeys $x(t)\to 0$ as $t\to\infty$ using only the global stability condition $xf(x)>0$  for $x\neq 0$. This condition is nothing other than the dissipative condition \eqref{eq.fglobalunperturbed} in one dimension. In this paper, a direct analogue of this result in the stochastic case is proven in Theorem~\ref{th.sigL2}, because the cumulative stochastic perturbation $\int_0^t \sigma(s)\,dB(s)$ converges when 
$\sigma\in L^2([0,\infty);\mathbb{R}^{d\times r})$. 

There is one final result in this section. It gives a complete
characterisation of the asymptotic behaviour of solutions of
\eqref{eq.mainstocheqn} under a strengthening of \eqref{eq.fasy},
namely
\begin{equation}\label{eq.fasybounded}
\liminf_{r\to\infty} \inf_{\|x\|=r}\frac{\langle
x,f(x)\rangle}{\|x\|}=+\infty.
\end{equation}
\eqref{eq.fasybounded} is a direct analogue of the condition needed
to give a classification of solutions of \eqref{eq.mainstocheqn} in
the scalar case. The following result is therefore a direct
generalisation of a scalar result from~\cite{JAAR:2010a} to finite
dimensions. 
\begin{theorem} \label{theorem.Xclassify}
Suppose $f$ obeys 
\eqref{eq.fglobalunperturbed} and \eqref{eq.fasybounded}. Suppose that $\sigma$ obeys \eqref{eq.sigmacns}.
 Suppose that $X$ is a continuous adapted process that obeys  \eqref{eq.mainstocheqn}.
 Then the following hold: 
\begin{itemize}
\item[(A)] If $S_h'$ obeys \eqref{eq.thetastableh}, then $\lim_{t\to\infty} X(t,\xi)=0$, a.s. for
each $\xi\in\mathbb{R}^d$.
\item[(B)] If $S_h'$ obeys \eqref{eq.thetaboundedh}, then there exists deterministic $0<c_1\leq c_2<+\infty$ such that
\[
c_1\leq \limsup_{t\to\infty} \|X(t,\xi)\|\leq c_2, \quad \liminf_{t\to\infty} \|X(t,\xi)\|=0, \quad \text{a.s., for
each $\xi\in\mathbb{R}^d$}. 
\]
Moreover, 
\begin{equation}\label{eq.aveXfXintto0}
\lim_{t\to\infty} \frac{1}{t}\int_0^t \langle
X(s),f(X(s))\rangle\,ds=0, \quad\text{a.s.}
\end{equation}
\item[(C)]  If $S_h'$ obeys \eqref{eq.thetaunstableh}, then $\limsup_{t\to\infty} \|X(t,\xi)\|=+\infty$ a.s., for
each $\xi\in\mathbb{R}^d$.
\end{itemize}
\end{theorem}
Before moving to the next section, we remark on the limit in \eqref{eq.aveXfXintto0}. Since $f$ obeys $\langle x,f(x)\rangle >0$ for all $x\neq 0$, 
\eqref{eq.aveXfXintto0} implies that despite $\|X\|$ assuming values bounded away from zero infinitely often, ``most of the time'' the process $\|X\|$ is close to zero rather than to its upper bounds.  

\subsection{Sufficient conditions on $\sigma$ for stability and asymptotic classification}
We deduce some conditions which are more easily verified than
\eqref{eq.thetastableh}, \eqref{eq.thetaboundedh} or \eqref{eq.thetaunstableh}.
In view of Theorem~\ref{th.sigL2}, in what
follows, we therefore concentrate on the case when $\sigma$ is not
in $L^2([0,\infty);\mathbb{R}^{d\times r})$. In this case, there
exists a pair of integers $(i,j)\in
\{1,\ldots,d\}\times\{1,\ldots,r\}$ such that $\sigma_{ij}\not\in
L^2([0,\infty);\mathbb{R})$.
Define
\begin{equation} \label{def.sigmai}
\sigma_i^2(t)=\sum_{l=1}^r \sigma_{il}^2(t), \quad t\geq 0.
 \end{equation}
 Then $\sigma_i\not\in L^2(0,\infty)$, and it is possible to define a number $T_i>0$ such that $\int_0^t e^{2s}\sigma_i^2(s)\,ds>e^e$
 for $t>T_i$ and so one can define a function $\Sigma_i:[T_i,\infty)\to [0,\infty)$ by
\begin{equation} \label{def.Sigmai}
\Sigma_i(t)=\left( \int_0^t e^{-2(t-s)}\sigma_i^2(s)\,ds
\right)^{1/2} \left(\log\log \int_0^t e^{2s}\sigma_i^2(s)\,ds
\right)^{1/2}, \quad t\geq T_i.
\end{equation}
The significance of the function $\Sigma_i$ defined in
\eqref{def.Sigmai} is that it characterises the largest possible
fluctuations of $Y_i(t)=\langle Y(t),\mathbf{e}_i\rangle$ for $i=1,\ldots,d$ when $\sigma_i$ is not square integrable.
\begin{equation} \label{eq.YiasySig} \limsup_{t\to\infty}
\frac{|Y_i(t)|}{\Sigma_i(t)}=\sqrt{2}, \quad\text{a.s.}
\end{equation}
This result follows by applying the Law of the iterated logarithm for martingales to $M(t):=\int_0^t e^s
\sigma_i(s)\,d\bar{B}_i(s)$. This holds because $\sigma_i\not \in
L^2([0,\infty);\mathbb{R}^{d\times r})$ implies that $\langle
M\rangle (t)=\int_0^t e^{2s} \sigma_i^2(s)\,ds\to\infty$ as
$t\to\infty$. 

Hence, by 
Theorem~\ref{th.xto0yto0}, the functions $\Sigma_i$ determine the asymptotic behaviour of $X$.
Let $N\subseteq \{1,2,\ldots,d\}$ be defined by
\begin{equation} \label{def.N}
N=\{i\in \{1,2,\ldots,d\}\,:\,\sigma_i\not \in L^2(0,\infty)\}.
\end{equation}
Note that if $i\not\in N$, then $\sigma_i\in L^2(0,\infty)$ and we
immediately have that $Y_i(t)\to 0$ as $t\to\infty$ a.s.
\begin{theorem} \label{th.sigmanotl2xto0}
Suppose that $f$ satisfies \eqref{eq.fglobalunperturbed}
and \eqref{eq.fasy}. Suppose that $\sigma$ obeys
\eqref{eq.sigmacns} and $\sigma\not\in
L^2([0,\infty);\mathbb{R}^{d\times r})$. Suppose that $X$ is a continuous adapted process which obeys
\eqref{eq.mainstocheqn}. Let $N$ be the set defined in \eqref{def.N}
and $\Sigma_i$ be defined by \eqref{def.Sigmai} for each $i\in N$.
\begin{itemize}
\item[(i)] If $\Sigma_i(t)\to 0$ as $t\to\infty$ for each $i\in N$, then $X$ obeys \eqref{eq.stochglobalstable}.
\item[(ii)] If $X$ obeys \eqref{eq.stochglobalstable}, then
$\liminf_{t\to\infty}\Sigma_i(t)=0$ for each $i\in N$.
\item[(iii)] If $\liminf_{t\to\infty}\Sigma_i(t)>0$ for some $i\in N$, then $\mathbb{P}[\lim_{t\to\infty} X(t)=0]=0$.
\item[(iv)] If $\lim_{t\to\infty}\Sigma_i(t)=\infty$ for some $i\in N$ then $\limsup_{t\to\infty} \|X(t)\|=\infty$ a.s.
\end{itemize}
\end{theorem}
%
%

In our next result we show that the asymptotic behaviour of the solution of \eqref{eq.mainstocheqn} can be classified
according as to whether a certain limit exists.
\begin{theorem} \label{theorem.XsiglogtL}
Suppose $f$ obeys 
\eqref{eq.fglobalunperturbed} and \eqref{eq.fasybounded}. Suppose that $\sigma$ obeys
\eqref{eq.sigmacns}.  Suppose that $X$ is a continuous adapted process that obeys  \eqref{eq.mainstocheqn}.
Suppose that there exists $h>0$ and $L_h\in[0,\infty]$ such that
\begin{equation} \label{eq.intsiglogntoLh}
\lim_{n\to\infty} \int_{nh}^{(n+1)h} \|\sigma(s)\|_F^2\,ds \cdot \log n =L_h.
\end{equation}
\begin{itemize}
\item[(i)] If $L_h=0$, then $\lim_{t\to\infty} X(t,\xi)=0$, a.s. for each $\xi\in\mathbb{R}^d$.
\item[(ii)] If $L_h\in (0,\infty)$, then there exist $0\leq c_1\leq c_2<\infty$ independent of $\xi$ such that
\[
c_2\leq \limsup_{t\to\infty} \|X(t,\xi)\|\leq c_2, \quad \liminf_{t\to\infty} \|X(t,\xi)\|=0, \quad \text{a.s.}
\]
for each $\xi\in\mathbb{R}^d$.
\item[(iii)] If $L_h=+\infty$, then $\limsup_{t\to\infty} \|X(t,\xi)\|=+\infty$ a.s., for
each $\xi\in\mathbb{R}^d$.
\end{itemize}
\end{theorem}
The result is a corollary of Theorem~\ref{theorem.Xclassify}, together with the observation that if $L_h=0$, then
$S_h'(\epsilon)<+\infty$ for all $\epsilon>0$; that $L_h\in (0,\infty)$ implies that there exists $\epsilon'>0$
such that $S_h'(\epsilon)<+\infty$ for all $\epsilon>\epsilon'$ and  $S_h'(\epsilon)=+\infty$ for all $\epsilon<\epsilon'$, while  $L_h=\infty$ implies that $S_h'(\epsilon)=+\infty$ for all $\epsilon>0$.

If pointwise conditions are preferred  to \eqref{eq.intsiglogntoLh} in Theorem~\ref{theorem.XsiglogtL}, we may instead impose the condition
\begin{equation}  \label{eq.siglogtoL}
\lim_{t\to\infty} \|\sigma(t)\|^2_F\log t=L\in [0,\infty]
\end{equation}
on $\sigma$. In this case, if $L=0$, then $L_h=0$ in \eqref{eq.intsiglogntoLh}, and part (i) of Theorem~\ref{theorem.XsiglogtL} applies; if $L\in (0,\infty)$, then $L_h=hL$ in \eqref{eq.intsiglogntoLh} and part (ii) of
Theorem~\ref{theorem.XsiglogtL} applies; and if $L=\infty$, then $L_h=+\infty$ in \eqref{eq.intsiglogntoLh},
and part (iii) of Theorem~\ref{theorem.XsiglogtL} applies.

If limits of the form \eqref{eq.intsiglogntoLh} or \eqref{eq.siglogtoL} do not exist, but appropriate limits inferior or superior are finite and bounded away from zero, some results on boundedness are still available. The following result is representative.
\begin{theorem} \label{theorem.stochastablegen}
Suppose that $f$ obeys \eqref{eq.fglobalunperturbed} 
 and \eqref{eq.fasybounded}. Suppose that $\sigma$ obeys
\eqref{eq.sigmacns}. Suppose that $X$ is a continuous adapted process that obeys \eqref{eq.mainstocheqn}.
\begin{itemize}
\item[(i)] If $\liminf_{t\to\infty} \|\sigma(t)\|^2_F\log t>0$, then $\limsup_{t\to\infty} \|X(t)\|\geq c_1$ a.s.
\item[(ii)] If $\limsup_{t\to\infty} \|\sigma(t)\|^2_F\log t<+\infty$, then
$\limsup_{t\to\infty} \|X(t)\|\leq c_2$ a.s.
\item[(iii)] If
\[
0<\liminf_{t\to\infty} \|\sigma(t)\|^2_F\log t\leq \limsup_{t\to\infty}  \|\sigma(t)\|^2_F\log t<+\infty,
\]
then
\[
0<c_1\leq \limsup_{t\to\infty} \|X(t)\|\leq c_2, \quad \text{a.s.}
\]
\end{itemize}
\end{theorem}
These conclusions follow from the observation that $\liminf_{t\to\infty} \|\sigma(t)\|_F^2\log t>0$
implies that $S_h'(\epsilon)=+\infty$ for all $\epsilon<\epsilon_1$ and that $\limsup_{t\to\infty}  \|\sigma(t)\|^2_F\log t<+\infty$ implies that
$S_h'(\epsilon)<+\infty$ for all $\epsilon>\epsilon_2$, in conjunction with part (B) of Theorem~\ref{theorem.Xclassify}.

In \cite{ChanWill:1989}, Chan and Williams have proven in the case
when $t\mapsto\sigma^2(t)$ is decreasing, that $Y$ obeys
\eqref{eq.Ytto0} if and only if $\sigma$ obeys
\eqref{eq.sigmalogto0}. Our final result shows that this pointwise monotonicity condition can be
weakened. Naturally, our conditions on $f$ are also weaker.
\begin{theorem} \label{theorem.eqvt}
Suppose that $f$  obeys \eqref{eq.fglobalunperturbed}
and \eqref{eq.fasy}. Suppose that $\sigma$ obeys
\eqref{eq.sigmacns} and 
that the sequence $n\mapsto \int_{nh}^{(n+1)h} \|\sigma(s)\|^2_F\,ds$ is non--increasing.
Suppose that $X$ is a continuous adapted process which obeys \eqref{eq.mainstocheqn}.
Then the following are equivalent:
\begin{itemize}
\item[(A)] $\sigma$ obeys $\lim_{n\to\infty} \int_{nh}^{(n+1)h} \|\sigma(s)\|^2_F\,ds \cdot \log n=0$;
\item[(B)] $\lim_{t\to\infty} X(t,\xi)=0$ with positive probability for some $\xi\in\mathbb{R}^d$;
\item[(C)] $\lim_{t\to\infty} X(t,\xi)=0$ a.s. for each $\xi\in\mathbb{R}^d$.
\end{itemize}
\end{theorem}
Stronger monotonicity conditions which can be imposed are that
\[
t\mapsto \Sigma^2_1(t)=\int_t^{t+1}\|\sigma(s)\|^2_F\,ds, \quad t\mapsto \Sigma^2_2(t)=\|\sigma(t)\|^2_F,
\]
are non--increasing. In this case statement (A) in Theorem~\ref{theorem.eqvt} can be replaced by
\[
\lim_{t\to\infty} \Sigma_i^2(t)\log t=0, \quad i=1,2.
\]
Theorem~\ref{theorem.eqvt} is proven by observing that when $n\mapsto \int_{nh}^{(n+1)h} \|\sigma(s)\|^2_F\,ds$ is non--increasing, then $\lim_{n\to\infty} \int_{nh}^{(n+1)h} \|\sigma(s)\|^2_F\,ds \cdot \log n=0$ is equivalent to
$S_h'(\epsilon)<+\infty$ for all $\epsilon>0$, which by Theorem~\ref{theorem.Xiffsigma}, is known to be equivalent to
statements (B) and (C).

\subsection{Deterministic analysis of global stability and the dissipative condition}
%
Before turning to the proofs of our results, we wish to comment on the dissipative condition \eqref{eq.fglobalunperturbed} and how it relates to 
hypotheses on $f$ made when establishing global asymptotic stability for  the ordinary differential equation \eqref{eq.unperturbed}.
The analysis of such good sufficient conditions on $f$ 
%
forms a substantial body of work, and rather than attempting to trace this, we mention the original contributions of Olech and Hartman in a series of papers in the 1960s.
In Hartman \cite{hart:1961}, global stability is assured by
\begin{equation} \label{eq.introch4h}
[J(x)]_{ij}=\frac{\partial f_i}{\partial x_j}(x) \text{ is such that $H(x):=\frac{1}{2}(J(x)+J(x)^T)$ is negative definite}
\end{equation}
In the two--dimensional case, Olech  \cite{olech:1960} proves that
\begin{equation} \label{eq.introch4o1}
\text{trace}{J(x)}\leq 0 \text{ and }  |f(x)|\geq \phi>0 \text{ for $|x|\geq x^\ast$}
\end{equation}
suffice. The second of these conditions is weakened in Hartman and Olech \cite{hartolech:1962} to
\begin{equation} \label{eq.introch4ho1}
|x||f(x)|>K \text{ for all $|x|\geq M$, or } \int_0^\infty \inf_{\|x\|=\rho} |f(x)|\,d\rho=+\infty
\end{equation}
and the first of Olech's assumptions is modified to
\begin{equation} \label{eq.introch4ho2}
\alpha(x)\leq 0, \text{ where } \alpha(x)=\max_{1\leq i< j\leq d}\{\lambda_i(x)+\lambda_j(x)\}
\end{equation}
and the $\lambda(x)$'s are eigenvalues of $H(x)$. The local asymptotic stability of the equilibrium is also assumed. In the 1970's Brock and Scheinkman~\cite{brockschein:1976} demonstrated that some of Olech and Hartman's conditions can be deduced from Liapunov considerations. In particular, they show that some of the conditions used in \cite{hart:1961} imply the dissipative condition. This is of particular interest to us, as our approach to understanding the stability and boundedness of solutions may be considered a Liapunov--like approach. A more recent paper of Gasull, LLibre and Sotomayor~\cite{gasullllibresoto:1991} considers the relationships between these conditions and global stability. As the paper develops, the relationship between
these existing conditions and the conditions we will need are drawn out.

\section{Proofs} 
\subsection{Proof of Proposition~\ref{prop.exist}}
Since $\sigma$ is continuous, there is a unique continuous adapted process which obeys
\[
dY(t)=-Y(t)\,dt + \sigma(t)\,dB(t), \quad t\geq 0; \quad Y(0)=0.
\]
Suppose that the a.s. event on which such a continuous adapted process is defined is $\Omega_Y$.
Consider now for $\omega\in \Omega_Y$ the parameterised random differential equation
\[
z'(t,\omega)=-f(z(t,\omega)+Y(t,\omega))+Y(t,\omega), \quad t>0; \quad z(0)=\xi.
\]
Since $f$ is continuous and the sample paths of $Y$ are continuous, there is a (local) solution $z(\cdot,\omega)$
for each $\omega\in \Omega_Y$ up to a time $\tau(\omega)\in (0,+\infty]$, where $\tau(\omega)=\inf\{t>0:z(t,\omega)\not\in(-\infty,\infty)\}$. Since $Y$ is adapted to the filtration generated by the standard Brownian motion $B$, and indeed is a functional of $B$, $z$ is also adapted to the filtration generated by $B$ and is a functional of $B$. Consider now for $\omega\in \Omega_Y$ the process $X$ defined by
\[
X(t,\omega)=z(t,\omega)+Y(t,\omega), \quad t\in [0,\tau(\omega)).
\]
The interval on which $X(\omega)$ is defined is the same as $z(\omega)$ because $Y(t,\omega)$ is finite for all finite $t$.
Moreover, it can be seen that $X(t)$ is a functional of $\{B(s):0\leq s\leq t\}$ because $z(t)$ and $Y(t)$ are. Let's define for $n\geq \lceil\|\xi\|\rceil=:n^\ast$ the time $\tau_n(\omega)=\inf\{t>0:\|X(t,\omega)\|=n\}$. Then on $\Omega_Y$, we see that $(\tau_n)_{n\geq n^\ast}$ is a sequence of stopping times adapted to the filtration generated by $B$. Moreover,
we have that $\tau_n$ is an increasing sequence with $\lim_{n\to\infty} \tau_n=\tau_\infty$, and $\tau_\infty(\omega)=\tau(\omega)$. Then for $t\geq 0$ we have
\begin{align*}
X(t\wedge \tau_n)&=z(t\wedge\tau_n)+Y(t\wedge \tau_n)\\
&=\xi-\int_0^{t\wedge \tau_n} f(z(s)+Y(s))\,ds + \int_0^{t\wedge \tau_n} Y(s)\,ds
-\int_0^{t\wedge \tau_n} Y(s)\,ds \\
&\qquad+ \int_0^{t\wedge \tau_n} \sigma(s)\,dB(s)\\
&=X(0)-\int_0^{t\wedge \tau_n} f(X(s))\,ds + \int_0^{t\wedge \tau_n}\sigma(s)\,dB(s)
\end{align*}
a.s. on $\Omega_Y$. Therefore, we have that $X$ is a solution to \eqref{eq.mainstocheqn}.

To demonstrate that any such solution is global, we proceed by a standard proof by contradiction.
Define $n^\ast\in\mathbb{N}$ such that $n^\ast>\|\xi\|$. Define for each $n\geq n^\ast$ the stopping time
$\tau_n^\xi=\inf\{t>0:\|X(t)\|_2=n\}$. We see that
$\tau_n^\xi$ is an increasing sequence of times and so
$\tau_\infty^\xi:=\lim_{n\to\infty} \tau_n^\xi$. Suppose, in
contradiction to the desired claim, that $\tau_\infty^\xi<+\infty$ with positive probability for some $\xi\in\mathbb{R}^d$.
Then, there exists $T>0$, $\epsilon>0$ and $n_0\in\mathbb{N}$ such that
\[
\mathbb{P}[\tau_n^\xi\leq T] \geq \epsilon, \quad n\geq
n_0>n^\ast.
\]
Consider now the non--negative semimartingale $\|X(t)\|^2_2$. Then by It\^o's rule
we have
\begin{multline}   \label{eq.Xsqstop}
\|X(t\wedge\tau_n^\xi)\|^2=\|\xi\|^2_2 - \int_0^{t\wedge \tau_n^\xi} \langle f(X(s)),X(s)\rangle\,ds \\
+ \int_0^{t\wedge \tau_n^\xi} \|\sigma(s)\|^2_F\,ds + M(t), \quad t\in [0,T],
\end{multline}
where
\[
M(t)=\sum_{j=1}^r \int_0^{t\wedge \tau_n^\xi} \left\{\sum_{i=1}^d 2X_i(s)\sigma_{ij}(s)\right\}\,dB_j(s), \quad t\in [0,T].
\]
By the optional sampling theorem, we have that $\mathbb{E}[M(t)]=0$, and so by \eqref{eq.fglobalunperturbed} we have that
\[
\mathbb{E}[
\|X(T\wedge\tau_n^\xi)\|_2^2]\leq \|\xi\|^2_2 + \int_0^T \|\sigma(s)\|^2_F\,ds =:K(T,\xi)<+\infty.
\]
Define next the event $C_n=\{\tau_n^\xi\leq T\}$. Then for $n\geq
n_0$ we have $\mathbb{P}[C_n]\geq \epsilon$. If $\omega\in C_n$, we
have that $\tau_n^\xi\leq T$, so $\|X(T\wedge \tau_n^\xi)\|_2=n$. Hence
$\|X(T\wedge \tau_n^\xi)\|^2_2=n^2$ for $\omega\in C_n$.
Hence
\begin{align*}
K(T,\xi)&\geq \mathbb{E}\left[\|X(T\wedge \tau_n^\xi)\|^2_2\right] \geq \mathbb{E}\left[\|X(T\wedge \tau_n^\xi)\|^2_2 I_{C_n}\right] =n^2 \mathbb{P}[C_n]\geq n^2 \epsilon.
\end{align*}
Therefore, we have that $K(T,\xi)\geq n^2 \epsilon$ for all
$n\geq n_0$. Letting $n\to\infty$ gives a contradiction.

\subsection{Proof of Theorem~\ref{th.sigL2}}
By It\^o's rule, and by virtue of the fact that $\|X(t)\|^2_2$ is finite for all $t\geq 0$ a.s., we may remove the stopping times 
in \eqref{eq.Xsqstop} above, and can write 
\begin{multline} \label{eq.Xsq}
\|X(t)\|^2_2=\|\xi\|^2_2 - \int_0^t 2\langle X(s),f(X(s))\rangle\,ds+\int_0^t \|\sigma(s)\|^2_F\,ds +M(t),
\quad t\geq 0,
\end{multline}
where we define $M$ to be the local martingale given by
\begin{equation} \label{def.Mmart}
M(t)=\sum_{j=1}^r \int_0^t \sum_{i=1}^d
2X_i(s)\sigma_{ij}(s)\,dB_j(s), \quad t\geq 0,
\end{equation}
and let
\[
U(t)=\int_0^t2\langle X(s),f(X(s))\rangle ds,\quad
A(t)=\int_0^t\|\sigma(s)\|^2_F\,ds,\quad t\geq 0.
\]
Since $\langle x,f(x)\rangle \geq 0$ for all $x\in\mathbb{R}^d$ and
$\sigma\in L^2([0,\infty);\mathbb{R}^{d\times r})$, it follows that
$A$ and $U$ are continuous adapted increasing processes. Therefore
by Lemma~\ref{theorem.convergence}, it follows that
\[
\lim_{t\to\infty}\|X(t)\|^2=L\in[0,\infty),\quad \text{a.s.}
\]
and that
\[
\lim_{t\to\infty}\int_0^t\langle X(s),f(X(s))\rangle
ds=I\in[0,\infty),\quad \text{a.s.}
\]
By continuity this means that there is an a.s. event $A=\{\omega:
\|X(t,\omega)\|\to \sqrt{L(\omega)}\in [0,\infty) \text{ as
$t\to\infty$}\}$. We write $A=A_+\cup A_0$ where
\begin{equation*}
A_+=\{\omega: \|X(t,\omega)\|\to \sqrt{L(\omega)}\in
(0,\infty)\text{ as $t\to\infty$}\},
\end{equation*}
and $A_0=\{\omega: X(t,\omega)\to 0 \text{ as $t\to\infty$}\}$.
Suppose that $\omega\in A_+$. Define
\[
F(x)=\langle x,f(x)\rangle, \quad x\in \mathbb{R}^d.
\]
By \eqref{eq.fglobalunperturbed}, we have that $F(x)=0$ if and only
if $x=0$. Define for any $r\geq 0$
\[
\inf_{\|x\|=r} F(x)=:\phi(r)\geq 0.
\]
Since $f$ is continuous and $F$ is continuous,  $\phi$ is
continuous. Hence $\min_{|x|=r} F(x)=\phi(r)$. Suppose there is
$r>0$ such that $\phi(r)=0$. Then there exists $x$ with $|x|=r$ such
that $F(x)=\phi(r)=0$. But this implies that $x=0$, a contradiction.
Moreover $\phi$ is continuous and positive definite. Hence for
$\omega\in A_+$ we have
\[
\liminf_{t\to\infty} \langle X(t,\omega),f(X(t,\omega))\geq
\phi(\sqrt{L(\omega)})>0.
\]
Therefore
\begin{equation} \label{eq.intxfxavemult}
\liminf_{t\to\infty} \frac{1}{t}\int_0^t \langle
X(s,\omega),f(X(s,\omega))\rangle\,ds \geq \phi(\sqrt{L(\omega)})>0.
\end{equation}
Since the last two terms on the righthand side of \eqref{eq.Xsq}
have finite limits as $t\to\infty$, \eqref{eq.intxfxavemult} implies
that for $\omega\in A_+$ that
\[
0\leq \lim_{t\to\infty}
\frac{\|X(t,\omega)\|^2}{t}=-2\phi(\sqrt{L(\omega)})<0,
\]
a contradiction. Therefore $\mathbb{P}[A_+]=0$. Since
$\mathbb{P}[A]=1$, we must have $\mathbb{P}[A_0]=1$, as required.

\subsection{Proof of Theorem~\ref{theorem.Xlowerboundunbound}}
Define
\begin{gather} \label{def.OmX}
\Omega_X=\bigl\{\omega\in \Omega: \text{there is a unique continuous adapted process $X$} \\
\text{for which the realisation $X(\cdot,\omega)$ obeys \eqref{eq.mainstocheqn}}\bigr\} \nonumber \\
\label{def.OmY}
\Omega_Y=\big\{\omega\in \Omega: \text{there is a unique continuous adapted process $Y$}\\
 \text{for which the realisation $Y(\cdot,\omega)$ obeys
 \eqref{def.Y}}\bigr\}. \nonumber
\end{gather}
Let \begin{equation} \label{def.Ome}
\Omega_e=\Omega_X\cap\Omega_Y.\end{equation} If $S_h'$ obeys
\eqref{eq.thetaunstableh}, it follows from
Theorem~\ref{theorem.Yclassify} that $\limsup_{t\to\infty}
\|Y(t)\|=+\infty$, a.s., and let the event on which this holds be
$\Omega_1\subseteq \Omega_Y$. Suppose that there is an event
$A=\{\omega:\limsup_{t\to\infty}  \|X(t,\omega)\|<+\infty\}$ for which
$\mathbb{P}[A]>0$. Define $A_1=A\cap \Omega_1\cap \Omega_e$ so that
$\mathbb{P}[A_1]>0$.

Next, rewrite \eqref{eq.mainstocheqn} as
\[
dX(t)=\left(-X(t)+[X(t)-f(X(t))]\right)\,dt + \sigma(t)\,dB(t),
\quad t\geq 0; \quad X(0)=\xi.
\]
Therefore on $\Omega_X$ we obtain
\[
X(t)=\xi e^{-t} + \int_0^t e^{-(t-s)}(X(s)-f(X(s)))\,ds +
e^{-t}\int_0^t e^{s}\sigma(s)\,dB(s).
\]
Since $Y$ obeys \eqref{eq.Yform}, for $\omega\in\Omega_e$ we have
\begin{equation} \label{eq.YX} Y(t,\omega)=X(t,\omega)-\xi e^{-t} -  \int_0^t
e^{-(t-s)}(X(s,\omega)-f(X(s,\omega)))\,ds, \quad t\geq 0.
\end{equation}
Define for $\omega\in A_1$
\begin{equation} \label{eq.limsupXmaxomega}
X^\ast(\omega):=\limsup_{t\to\infty} \|X(t,\omega)\|<+\infty,
\end{equation}
and define $\bar{f}(x)=\sup_{\|y\|\leq x} \|f(y)\|$ and
$\bar{F}(x)=2x+\bar{f}(x)$ for $x\geq 0$. Then for each $\omega\in
A_1$, it follows from \eqref{eq.YX} that
\[
\limsup_{t\to\infty}\|Y(t,\omega)\|\leq 2X^\ast(\omega) +
\bar{f}(X^\ast(\omega))=\bar{F}(X^\ast(\omega)),
\]
and as $\bar{F}(X^\ast(\omega))<+\infty$, a contradiction results.

To prove part (B), first note that $\bar{F}$ is continuous and
increasing on $[0,\infty)$ with $\bar{F}(0)=0$ and
$\lim_{x\to\infty} \bar{F}(x)=+\infty$. Therefore, for every $c>0$
there exists a unique $c'>0$ such that $\bar{F}(c)=c'$, or
$c'=\bar{F}^{-1}(c)$. Suppose that $S_h'$ obeys
\eqref{eq.thetaboundedh}, so that by Theorem~\ref{theorem.Yclassify}
there is a $c_1>0$ such that $\limsup_{t\to\infty} \|Y(t)\|\geq c_1$
a.s. Let the event on which this holds be $\Omega_2$. Suppose now
that the event $A_2$ defined by
\[
A_2=\{\omega\in \Omega_X: \limsup_{t\to\infty}
\|X(t,\omega)\|<\bar{F}^{-1}(c_1)\},
\]
and suppose that $\mathbb{P}[A_2]>0$. Define $A_3=A_2\cap
\Omega_e\cap \Omega_2$. Then $\mathbb{P}[A_3]>0$. For $\omega\in
A_3$, $X^\ast(\omega)$ as given by \eqref{eq.limsupXmaxomega} is
well--defined and finite, and in fact $X^\ast(\omega)<
\bar{F}^{-1}(c_1)$. As before, from \eqref{eq.YX}, we deduce that
$\limsup_{t\to\infty}\|Y(t,\omega)\|\leq \bar{F}(X^\ast(\omega))$.
But then we have $c_1\leq \bar{F}(X^\ast(\omega))$, which implies
$\bar{F}^{-1}(c_1)\leq X^\ast(\omega)<\bar{F}^{-1}(c_1)$, a
contradiction. Thus we have that $\mathbb{P}[A_2]=0$, so
$\limsup_{t\to\infty} \|X(t)\|\geq \bar{F}^{-1}(c_1)=:c_3>0$ a.s.,
as required.
\subsection{Proof of Theorem~\ref{th.xto0yto0}}
In this proof, we implicitly consider the case where $\sigma\not \in
L^2([0,\infty);\mathbb{R}^{d\times r})$, as Theorem~\ref{th.sigL2}
shows that the result holds in the case where $\sigma\in
L^2([0,\infty);\mathbb{R}^{d\times r})$, with each of the events
$\{\omega:\lim_{t\to\infty} Y(t,\omega)=0\}$ and
$\{\omega:\lim_{t\to\infty} X(t,\omega)=0\}$ being a.s.


We prove that $X(t)\to 0$ as $t\to\infty$ implies $Y(t)\to 0$ as
$t\to\infty$ i.e., \eqref{eq.Y0impX0}.
Since $f$ obeys \eqref{eq.fglobalunperturbed} it follows from
\eqref{eq.YX} that for each $\omega\in \{\omega:X(t,\omega)\to 0
\text{ as $t\to\infty$}\}\cap \Omega_e$ that $Y(t,\omega)\to 0$ as
$t\to\infty$, proving \eqref{eq.Y0impX0}.

We now prove that $Y(t)\to 0$ as $t\to\infty$ implies $X(t)\to 0$ as
$t\to\infty$ or $\|X(t)\|\to \infty$ as $t\to\infty$, i.e.
 \eqref{eq.X0impY0}.

Define $\Omega_2=\{\omega:\lim_{t\to\infty}
Y(t,\omega)=0\}\cap\Omega_Y$ and
\begin{gather*}
A_0=\{\omega: \liminf_{t\to\infty} \|X(t,\omega)\|=0\}, \\
A_+= \{\omega: \liminf_{t\to\infty} \|X(t,\omega)\|\in (0,\infty)\}, \\
A_\infty =\{\omega: \liminf_{t\to\infty} \|X(t,\omega)\|=\infty\}.
\end{gather*}
Also define
\begin{gather*}
\Omega_0=\Omega_2\cap\Omega_X\cap A_0=\Omega_2\cap\Omega_e\cap A_0, \\
\Omega_+=\Omega_2\cap\Omega_X\cap A_+=\Omega_2\cap \Omega_e\cap A_+, \\
\Omega_\infty=\Omega_2\cap\Omega_X\cap A_\infty=\Omega_2\cap
\Omega_e\cap A_\infty.
\end{gather*}
Finally define $A_1=\{\omega: \lim_{t\to\infty} X(t,\omega)=0\}$ and
$\Omega_1=\Omega_2\cap \Omega_X\cap A_1$. Clearly $A_1\subseteq A_0$
and $\Omega_1\subseteq \Omega_0$.

Define for each $\omega\in \Omega_e$ the realisation
$z(\cdot,\omega)$ by $z(t,\omega)=X(t,\omega)-Y(t,\omega)$ for
$t\geq 0$. Then $z(\cdot,\omega)$ is in $C^1(0,\infty)$ and obeys
\begin{equation*}
z'(t,\omega)=-f(X(t,\omega))+Y(t,\omega)=-f(z(t,\omega)+Y(t,\omega))+Y(t,\omega),
\, t\geq 0; \, z(0)=\xi.
\end{equation*}
Let $\omega\in \Omega_0\cup\Omega_+$. Then $\liminf_{t\to\infty}
\|X(t,\omega)\|<+\infty$. Define also
\[
g(t,\omega)=f(z(t,\omega))-f(z(t,\omega)+Y(t,\omega))+Y(t,\omega),
\quad t\geq 0.
\]
Since $z(\cdot,\omega)$ is in $C^1(0,\infty)$ we have
\begin{align*}
\frac{d}{dt}\|z(t,\omega)\|^2 &= 2\langle z(t,\omega),z'(t,\omega)\rangle \\
&=2\langle z(t,\omega), -f(z(t,\omega)) + f(z(t,\omega)) -f(z(t,\omega)+Y(t,\omega))+Y(t,\omega)\rangle \\
&=-2\langle z(t,\omega), f(z(t,\omega))\rangle + 2\langle
z(t,\omega), g(t,\omega)\rangle.
\end{align*}
Since $Y(t,\omega)\to 0$ as $t\to\infty$ and $\liminf_{t\to\infty}
\|X(t,\omega)\|=:L(\omega)<+\infty$, it follows that
\begin{align*}
\liminf_{t\to\infty} \|z(t,\omega)\|&\leq \liminf_{t\to\infty} \|X(t,\omega)\|+\|Y(t,\omega)\|\\
&=\liminf_{t\to\infty} \|X(t,\omega)\|+\lim_{t\to\infty}
\|Y(t,\omega)\|=L(\omega).
\end{align*}
Define $\lambda(\omega):=\liminf_{t\to\infty} \|z(t,\omega)\|$. Then
$\lambda(\omega)<+\infty$.
We remark that as $f$ is continuous, by the Heine--Cantor theorem it is uniformly continuous on compact sets. Therefore, for every fixed
$c>0$ we may define the a modulus of continuity $\omega'_c:[0,2c]\to \mathbb{R}^+$ for $f$ by
\[
\omega'_c(\delta):=\sup_{\|x\|\vee \|y\|\leq c, \|x-y\|= \delta} \|f(x)-f(y)\|.
\]
Define now $\omega_c(\delta):=\sup_{0\leq x\leq \delta} \omega_c'(x)$. Then $\omega_c$ is non--decreasing and we have
\[
\|f(x)-f(y)\|\leq \omega_c(\|x-y\|)\leq \omega_c(\delta) \text{ for all $\|x\|\vee\|y\|\leq c$ such that $\|x-y\|\leq \delta$}
\]
The uniform continuity of $f$ guarantees that $\omega_c(\delta)\to 0$ as $\delta\to 0$.

\textbf{STEP A:} We now show that $\liminf_{t\to\infty}
\|z(t,\omega)\|>0$ implies
\[
\limsup_{t\to\infty} \|z(t,\omega)\|<+\infty.
\]

\textbf{Proof of STEP A:} Suppose $\lambda(\omega)>0$ and
$\limsup_{t\to\infty} \|z(t,\omega)\|=+\infty$. Since $f$ is
continuous, and $\langle x,f(x)\rangle >0$ for $x\neq 0$, it follows
that there exists $F_\lambda>0$ such that
\[
F_\lambda:=\inf_{\|z\|=3\lambda/2} \langle z,f(z)\rangle.
\]
Also, using the modulus of continuity of $f$, we have that
\[
\|f(x)-f(y)\|\leq \omega_{3\lambda}(\|x-y\|), \quad \text{for all $\|x\|\vee
\|y\|\leq 3\lambda$}.
\]
Since $\omega_{3\lambda}(\delta)\to 0$ as $\delta\to 0$, we may choose $\epsilon>0$ so small that
\[
\epsilon<\frac{3\lambda(\omega)}{2}, \quad \epsilon+\omega_{3\lambda}(\epsilon)< \frac{2F_{\lambda(\omega)}}{3\lambda}.
\]
Since $Y(t,\omega)\to 0$ as $t\to\infty$, there exists
$T_1(\epsilon,\omega)>0$ such that $\|Y(t,\omega)\|<\epsilon$ for
all $t>T_1(\epsilon,\omega)$. Suppose that
\[
\limsup_{t\to\infty} \|z(t,\omega)\|=+\infty.
\]
Then there exists $T_2(\epsilon)>T_1(\epsilon)$ such that
$T_2(\epsilon)=\inf\{t>T_1(\epsilon)\,:\, \|z(t)\|=3\lambda/2\}$.
Define also
\[
T_3(\epsilon)=\inf\{t>T_2(\epsilon)\,:\, \|z(t)\|=5\lambda/4\},
\quad T_4(\epsilon)=\inf\{t>T_3(\epsilon)\,:\,
\|z(t)\|=3\lambda/2\}.
\]
Clearly with $w(t)=\|z(t,\omega)\|^2$, we have $w'(T_3,\omega)\leq
0$ and $w'(T_4,\omega)\geq 0$. Since $z(T_4)=3\lambda/2$ we have
$\langle z(T_4),f(z(T_4))\rangle \geq F_\lambda$. Also we have
$\|z(T_4)+Y(T_4)\|\leq \|z(T_4)\|+\|Y(T_4)\|\leq
3\lambda/2+\epsilon\leq 3\lambda$, so
\[
\|f(z(T_4)+Y(T_4))-f(z(T_4))\|\leq \omega_{3\lambda}(\|Y(T_4)\|)\leq \omega_{3\lambda}(\epsilon).
\]
Collecting these estimates yields
\begin{align*}
\lefteqn{w'(T_4)}\\
&=-2\langle z(T_4),f(z(T_4))\rangle + 2\langle z(T_4),g(T_4)\rangle \\
&=-2\langle z(T_4),f(z(T_4))\rangle + 2\langle z(T_4), f(z(T_4))-f(z(T_4)+Y(T_4))+Y(T_4)\rangle \\
&\leq -2F_\lambda + 2\cdot \frac{3\lambda}{2} \epsilon + 2\frac{3\lambda}{2}\|f(z(T_4))-f(z(T_4)+Y(T_4)))\| \\
&\leq -2F_\lambda +  3\lambda \epsilon + 3\lambda \omega_{3\lambda}(\epsilon) <0.
\end{align*}
Therefore we have a contradiction, because $w'(T_4)\geq 0$.

\textbf{STEP B:} Next we show that $\liminf_{t\to\infty}
\|z(t,\omega)\|=0$ implies
\[
\limsup_{t\to\infty} \|z(t,\omega)\|<+\infty.
\]

\textbf{Proof of STEP B:} Suppose to the contrary that
$\limsup_{t\to\infty} \|z(t,\omega)\|=+\infty$. Fix $\lambda>0$
arbitrarily. Proceeding exactly as in STEP A, we can demonstrate
that the supposition $\limsup_{t\to\infty} \|z(t,\omega)\|=\infty$
leads to a contradiction. Therefore we have shown that
$\liminf_{t\to\infty} \|z(t,\omega)\|\in [0,\infty)$ implies that
$\limsup_{t\to\infty} \|z(t,\omega)\|<+\infty$.

\textbf{STEP C:} Next we show that
\[
\liminf_{t\to\infty} \|X(t,\omega)\|<+\infty
\]
implies that $\liminf_{t\to\infty} \|z(t,\omega)\|=0$,
$\limsup_{t\to\infty} \|z(t,\omega)\|<+\infty$.

\textbf{Proof of STEP C:} First, we note that $\liminf_{t\to\infty}
\|X(t,\omega)\|<+\infty$ implies that $\liminf_{t\to\infty}
\|z(t,\omega)\|<+\infty$. By STEPs A and B, implies
$\limsup_{t\to\infty} \|z(t,\omega)\|<+\infty$. Define
\[
\limsup_{t\to\infty} \|z(t,\omega)\|=:\Lambda'(\omega)\in
[0,\infty).
\]
Suppose that $\liminf_{t\to\infty}
\|z(t,\omega)\|=\lambda(\omega)>0$. Then $\Lambda'\geq \lambda>0$.
By the continuity of $f$, the fact that $\Lambda'\geq \lambda>0$,
and the fact that $f$ obeys $\langle x,f(x)\rangle >0$ for all
$x\neq 0$, there exists an $F_{\lambda,\Lambda'}>0$ defined by
\[
F_{\lambda(\omega),\Lambda'(\omega)}:=\min_{\lambda(\omega)/2\leq
\|x\|\leq \Lambda'(\omega)+\lambda(\omega)/2} \langle x,f(x)\rangle.
\]
Suppose now that $\epsilon>0$ is so small that
\[
0<\epsilon< \frac{\lambda(\omega)}{2}, \quad \epsilon+\omega_{\Lambda'+\lambda}(\epsilon)<\frac{F_{\lambda(\omega),\Lambda'(\omega)}}{2(\Lambda'(\omega)+\lambda(\omega)/2)}.
\]
Then there exists $T_1(\epsilon,\omega)>0$ such that
$\|Y(t,\omega)\|<\epsilon$ for all $t>T_1(\epsilon,\omega)$. Also,
there exists $T_2(\omega)>0$ such that $\|z(t,\omega)\|\leq
\Lambda'(\omega)+\lambda(\omega)/2$ for all $t\geq T_2(\omega)$.
Now let $T_3(\epsilon,\omega)=1+T_1(\epsilon,\omega)\vee
T_2(\omega)$. Then for $t\geq T_3(\epsilon,\omega)$ we have
$\|z(t,\omega)+Y(t,\omega)\|\leq
\Lambda'(\omega)+\lambda(\omega)/2+\epsilon
<\Lambda'(\omega)+\lambda(\omega)$ 
and $\|z(t,\omega)\|\leq \Lambda'(\omega)+\lambda(\omega)$. Therefore for $t\geq
T_3(\epsilon,\omega)$ we have
\[
\|f(z(t,\omega)+Y(t,\omega))-f(z(t,\omega))\| \leq \omega_{\Lambda'+\lambda}(\|Y(t,\omega)\|)\leq \omega_{\Lambda'+\lambda}(\epsilon),
\]
and hence
\begin{align*}
\lefteqn{|\langle g(t,\omega),z(t,\omega) \rangle| }\\
&\leq \|z(t,\omega)\|\|f(z(t,\omega)+Y(t,\omega))-f(z(t,\omega))\| +|\langle z(t,\omega),Y(t,\omega)\rangle|\\
&\leq \omega_{\Lambda'+\lambda}(\epsilon) \|z(t,\omega)\| + \|z(t,\omega)\|\|Y(t,\omega)\| \\
&\leq (\omega_{\Lambda'+\lambda}(\epsilon) + \epsilon)(\Lambda'+\lambda/2)\\
&<F_{\lambda,\Lambda'}/2.
\end{align*}
Since $\liminf_{t\to\infty} \|z(t,\omega)\|=\lambda(\omega)>0$ there
exists $T_4(\omega)>0$ such that $\|z(t,\omega)\|>\lambda(\omega)/2$
for all $t\geq T_4(\omega)$. Define
$T_5(\epsilon,\omega)=1+T_4(\omega)\vee T_3(\epsilon,\omega)$. Then
for $t\geq T_5(\epsilon,\omega)$ we have
$0<\lambda(\omega)/2<\|z(t,\omega)\|\leq
\Lambda'(\omega)+\lambda(\omega)/2$, which implies that
\[
\langle z(t,\omega),f(z(t,\omega))\rangle \geq
F_{\lambda,\Lambda'}>0.
\]
Therefore for $t\geq T_5(\epsilon,\omega)$ we have
\begin{align*}
\frac{d}{dt}\|z(t,\omega)\|^2 &= -2\langle z(t,\omega),f(z(t,\omega))\rangle + 2\langle g(t,\omega),z(t,\omega) \rangle \\
&\leq -2\langle z(t,\omega),f(z(t,\omega))\rangle + F_{\lambda,\Lambda'} \\ 
&\leq - F_{\lambda,\Lambda'}.
\end{align*}
Therefore for $t\geq T_5(\epsilon,\omega)$ we have
\[
\|z(t,\omega)\|^2\leq \|z(T_5)\|^2 - F_{\lambda,\Lambda'}(t-T_5).
\]
Hence we have that $\|z(t,\omega)\|^2\to -\infty$ as $t\to\infty$,
which is a contradiction. Thus $\liminf_{t\to\infty}
\|z(t,\omega)\|=0$, as required.

\textbf{STEP D:} Suppose that
\[
\liminf_{t\to\infty} \|X(t,\omega)\|<+\infty.
\]
Then $\lim_{t\to\infty} X(t,\omega)=0$.

\textbf{Proof of STEP D:} By STEP C, $\liminf_{t\to\infty}
\|X(t,\omega)\|<+\infty$, this implies that $\liminf_{t\to\infty}
\|z(t,\omega)\|=0$ and $\limsup_{t\to\infty}
\|z(t,\omega)\|<+\infty$. If we can show that
\[
\lim_{t\to\infty} \|z(t,\omega)\|=0,
\]
we are done because $X(t,\omega)=z(t,\omega)+Y(t,\omega)$ and
$Y(t,\omega)\to 0$ as $t\to\infty$.
Let $\eta>0$. We next show that $\limsup_{t\to\infty}
\|z(t,\omega)\|\leq \eta$. Using the modulus of continuity of $f$ we have that
\[
\|f(x)-f(y)\|\leq \omega_{2\eta}(\|x-y\|)\leq \omega_{2\eta}(\delta) \text{ for all $\|x\|\vee\|y\|\leq 2\eta$, $\|x-y\|\leq \delta\leq 4\eta$}.
\]
There also exists $F_\eta>0$ such that
\[
F_\eta:=\min_{\|x\|=\eta} \langle x,f(x)\rangle.
\]
Let $\epsilon>0$ be so small that
\[
\epsilon<\frac{\eta}{2}, \quad \epsilon+\omega_{2\eta}(\epsilon)< \frac{F_\eta}{\eta}.
\]
Since $Y(t,\omega)\to 0$ as $t\to\infty$, there exists
$T_1(\epsilon,\omega)>0$ such that $\|Y(t,\omega)\|<\epsilon$ for
all $t>T_1(\epsilon)$. Suppose that $\limsup_{t\to\infty}
\|z(t,\omega)\|>\eta$. Since $\liminf_{t\to\infty}
\|z(t,\omega)\|=0$, we may therefore define
\begin{align*}
T_2(\epsilon,\omega)=\inf\{t>T_1(\epsilon,\omega):\|z(t,\omega)\|=\eta/2\}, \\
T_3(\epsilon,\omega)=\inf\{t>T_2(\epsilon,\omega):\|z(t,\omega)\|=\eta\}.
\end{align*}
Therefore, with $w(t)=\|z(t,\omega)\|^2$ we have that
$w'(T_3(\epsilon,\omega))\geq 0$. Furthermore, for $t\in
[T_2(\epsilon,\omega),T_3(\epsilon,\omega)]$ we have
$\|z(t,\omega)\|\leq \eta$ and $\|z(t,\omega)+Y(t,\omega)\|\leq
\eta+\epsilon<2\eta$ so
\begin{align*}
\|g(t,\omega)\|&\leq
\|f(z(t,\omega))-f(z(t,\omega)+Y(t,\omega))\|+\|Y(t,\omega)\|\\
&\leq \omega_{2\eta}(\|Y(t,\omega)\|)+\epsilon \leq \omega_{2\eta}(\epsilon)+\epsilon.
\end{align*}
Thus as  $\|z(T_3)\|=\eta$, we have
\[
|\langle  z(T_3),g(T_3)\rangle|\leq
\|z(T_3)\|\|g(T_3)\|=\eta\|g(T_3)\| \leq \eta (\omega_{2\eta}(\epsilon)+\epsilon)<F_\eta. 
\]
Since $\|z(T_3)\|=\eta$, we have $\langle z(T_3),f(z(T_3))\geq
F_\eta$ so therefore we have the estimate
\begin{equation*}
w'(T_3(\epsilon,\omega))=-2\langle z(T_3),f(z(T_3))\rangle + 2\langle z(T_3), g(T_3)\rangle
\leq -F_\eta<0,
\end{equation*}
a contradiction. Hence $T_3(\epsilon,\omega)$ does not exist for any
$\omega\in \Omega_0\cup \Omega_+$. Therefore we have
$\limsup_{t\to\infty} \|z(t,\omega)\|\leq \eta$. Since $\eta>0$ is
arbitrary, we make take the limit as $\eta\downarrow 0$ to obtain
$\limsup_{t\to\infty} \|z(t,\omega)\|=0$. Since $X=Y+z$, and
$Y(t,\omega)\to 0$ as $t\to\infty$, we have that $X(t,\omega)\to 0$
as $t\to\infty$.

\subsection{Proof of Theorem~\ref{theorem.Xiffsigma}}
Let $Y$ be the solution of \eqref{def.Y}. We prove first that
\eqref{eq.thetastableh} implies \eqref{eq.stochglobalstable}. First,
from Theorem~\ref{theorem.Yclassify}, we have that
\eqref{eq.thetastableh} implies $Y(t)\to 0$ as $t\to\infty$ a.s.
Moreover, if \eqref{eq.thetastableh} holds it follows that
\[
\sum_{n=0}^\infty \sqrt{\int_{nh}^{(n+1)h} \|\sigma(s)\|^2_F\,ds}  \exp\left(-\frac{\epsilon^2}{2}\frac{1}{\int_{nh}^{(n+1)h} \|\sigma(s)\|^2_F\,ds}\right)                    <+\infty \quad
\text{for each $\epsilon>0$}
\]
Therefore it follows that the summand tends to zero as $n\to\infty$, and so
\[
\lim_{n\to\infty} \int_{nh}^{(n+1)h} \|\sigma(s)\|^2_F\,ds=0.
\]
%
For every
$t> 0$ there is $n\in\mathbb{N}_0$ such that $t\in [nh,(n+1)h]$. Now
\begin{align*}
 \frac{1}{t}\int_0^t \|\sigma(s)\|^2_F\,ds &\leq \frac{1}{nh}\int_0^{(n+1)h} \|\sigma(s)\|^2_F\,ds 
 =
 \frac{1}{h}\cdot \frac{1}{n}\sum_{l=0}^n \int_{lh}^{(l+1)h} \|\sigma(s)\|^2_F\,ds.
 \end{align*}
 Since the summand tends to zero as $l\to\infty$, we have 
\begin{equation} \label{eq.cesarosigma0}
\lim_{t\to\infty} \frac{1}{t}\int_0^t \|\sigma(s)\|^2_F\,ds=0.
\end{equation}

Define the event $A=\{\omega: \|X(t,\omega)\|\to \infty \text{ as
$t\to\infty$}\}$. We prove that $\mathbb{P}[A]=0$. Suppose to the
contrary that $\mathbb{P}[A]>0$. Define
$\Omega_3=\Omega_2\cap\Omega_X\cap A$. Then by assumption
$\mathbb{P}[\Omega_3]>0$. By \eqref{eq.Xsq} we have
\begin{equation} \label{eq.Xforboundedtoo}
\|X(t)\|^2=\|\xi\|^2 - \int_0^t 2\langle
X(s),f(X(s))\rangle\,ds+\int_0^t \|\sigma(s)\|^2_F\,ds + M(t),
\quad t\geq 0.
\end{equation}
where $M$  is the local (scalar) martingale given by
\begin{equation}\label{eq.Mforboundedtoo}
M(t)=2\sum_{j=1}^r \int_0^t \sum_{i=1}^d
X_i(s)\sigma_{ij}(s)\,dB_j(s), \quad t\geq 0.
\end{equation}
Since $f$ obeys \eqref{eq.fasy}, i.e.,
\begin{equation*}
\liminf_{r\to\infty} \inf_{\|x\|=r} \langle x,f(x)\rangle=:\lambda>0,
\end{equation*}
for $\omega\in \Omega_3$ we have that
\[
\liminf_{s\to\infty} \langle X(s,\omega),f(X(s,\omega))\geq \lambda,
\]
so
\[
\liminf_{t\to\infty} \frac{2}{t}\int_0^t \langle
X(s,\omega),f(X(s,\omega))\rangle\,ds \geq 2\lambda,
\]
so for each $\epsilon<\lambda/3$, there exists
$T_1(\epsilon,\omega)>0$ such that
\[
\frac{2}{t}\int_0^t \langle X(s,\omega),f(X(s,\omega))\rangle\,ds
\geq 2\lambda-\epsilon, \quad t\geq T_1(\epsilon,\omega).
\]
By \eqref{eq.cesarosigma0}, for every $\epsilon>0$ there is
$T_2(\epsilon)>0$ such that
\[
\frac{\|\xi\|^2}{t}<\epsilon, \quad \frac{1}{t}\int_0^t
\|\sigma(s)\|^2_F\,ds<\epsilon, \quad t>T_2(\epsilon).
\]
Let $T(\epsilon,\omega)=1+T_1(\epsilon,\omega)\vee T_2(\epsilon)$.

Suppose there is a subevent $A'$ of $A$ with $\mathbb{P}[A']>0$ such
that $\langle M\rangle (t,\omega)\to\infty$ as $t\to\infty$ for each
$\omega\in A'$. Then $\liminf_{t\to\infty} M(t,\omega)=-\infty$ and
$\limsup_{t\to\infty} M(t,\omega)=+\infty$ for each $\omega\in A'$.
Then by the continuity of $M$ there exists
$\tau(\omega)>T(\epsilon,\omega)$ such that $M(\tau(\omega))=0$. Let
$t\geq T(\epsilon,\omega)$. Then
\begin{align*}
\frac{\|X(t,\omega)\|^2}{t}
&=\frac{\|\xi\|^2}{t} - 2\frac{1}{t}\int_0^t \langle X(s,\omega),f(X(s,\omega))\rangle\,ds+\frac{\int_0^t \|\sigma(s)\|^2_F\,ds}{t} + \frac{M(t,\omega)}{t}\\
&\leq \epsilon -2\lambda+\epsilon + \epsilon + \frac{M(t,\omega)}{t}\\
&=-2\lambda+3\epsilon+\frac{M(t,\omega)}{t}<-\lambda+\frac{M(t,\omega)}{t}.
\end{align*}
Hence
\[
0\leq \frac{\|X(\tau(\omega))\|^2}{\tau(\omega)} <
-\lambda+\frac{M(\tau(\omega))}{\tau(\omega)}= -\lambda<0,
\]
a contradiction. Therefore we have that $\lim_{t\to\infty} \langle
M\rangle (t)<+\infty$ a.s. on $A$. Hence $M(t)$ tends to a limit as
$t\to\infty$ a.s. on $A$ and so $M(t)/t\to 0$ as $t\to\infty$ a.s.
on $A$. Therefore,
\begin{align*}
\lefteqn{\limsup_{t\to\infty}
\frac{\|X(t,\omega)\|^2}{t}}\\
&=\limsup_{t\to\infty} \frac{\|\xi\|^2}{t} - \frac{2}{t}\int_0^t \langle X(s,\omega),f(X(s,\omega))\rangle\,ds+\frac{1}{t}\int_0^t \|\sigma(s)\|^2_F\,ds + \frac{M(t,\omega)}{t}\\
&=\limsup_{t\to\infty}  - 2\frac{1}{t}\int_0^t \langle X(s,\omega),f(X(s,\omega))\rangle\,ds\\
&=-2\liminf_{t\to\infty}  \frac{1}{t}\int_0^t \langle
X(s,\omega),f(X(s,\omega))\rangle\,ds \leq -2\lambda<0,
\end{align*}
a contradiction. Therefore, we must have $\mathbb{P}[A]=0$. Thus by
Theorem~\ref{th.xto0yto0}, it follows that $X(t)\to 0$ as
$t\to\infty$ a.s. We have shown that statement (A) and (C) are
equivalent.

Statement (C) implies statement (B). It remains to show that
statement (B) implies statement (A). By Theorem~\ref{th.xto0yto0},
it follows that $\mathbb{P}[Y(t)\to 0 \text{ as $t\to\infty$}]>0$.
Therefore by Theorem~\ref{theorem.Yclassify} it follows that
\eqref{eq.thetastableh} (or statement (A)) holds. Thus (C) implies
(B) implies (A).

\section{Proof of Theorem~\ref{theorem.Xclassify}}
We start by noticing that parts (A) and (C) of the theorem have
already been proven; part (A) is a consequence of
Theorem~\ref{theorem.Xiffsigma}, while part (C) is part (A) of
Theorem~\ref{theorem.Xlowerboundunbound}. The lower bound in part
(B) is a result of part (B) from
Theorem~\ref{theorem.Xlowerboundunbound}.

Therefore, it remains to establish the upper bound in part (B).
However, the proof of this result is technical, and relies on a
number of subsidiary results. The main step is a comparison theorem,
in which $\|X\|$ is bounded by the above by the positive solution of
$Z$ of a scalar stochastic differential equation. The solution of the scalar stochastic differential 
equation is then shown to be bounded by pathwise methods.

\subsection{Auxiliary functions and processes}
We start by introducing some functions and processes and deducing
some of their important properties. Let $\phi:[0,\infty)\to
\mathbb{R}$ be defined by
\begin{equation} \label{def.phimult}
\phi(x)=\inf_{\|y\|=x} \frac{\langle y,f(y)\rangle}{\|y\|}, \quad x>0;
\qquad \phi(0)=0.
\end{equation}
Since $f$ obeys \eqref{eq.fglobalunperturbed} it follows that
$\phi:[0,\infty)\to[0,\infty)$. Notice that $f$ being continuous ensures that $\phi\in C([0,\infty);[0,\infty))$. 
We now define 
\begin{equation} \label{def.phi0}
\phi_0(x)=\inf_{x/2\leq y\leq 4x} \phi(y), \quad x\geq 0, 
\end{equation}
and $\phi_1:[0,\infty)\to\mathbb{R}$ by $\phi_1(0)=0$ and 
\begin{equation} \label{def.phi1}
\phi_1(x)=\frac{1}{x}\int_x^{2x} (v\wedge 1) \phi_0(v)\,dv, \quad x>0.
\end{equation}
The motivation behind the construction of the function $\phi_1$ is to produce a \emph{Lipschitz} continuous function 
which shares the properties listed in \eqref{eq.phi1pos} with $\phi$, but bounds $\phi$ below. The Lipschitz continuity is important, because it ensures that a certain stochastic differential equation will have a unique solution; the fact that it bounds $\phi$ below
means that we will be able to prove, via a comparison approach, that the solution of the stochastic differential equation dominates $\|X\|$. 
\begin{lemma} \label{lemma.phi1loclip}
Suppose that $f$ obeys \eqref{eq.fglobalunperturbed} and \eqref{eq.fasybounded}. Then $\phi_1$ defined by \eqref{def.phi1} 
is locally Lipschitz continuous on $[0,\infty)$, 
\begin{equation}  \label{eq.phi1ltphi}
\phi_1(x)\leq \phi(x), \quad x\geq 0,
\end{equation}
and also obeys
\begin{equation}  \label{eq.phi1pos}
\phi_1(x)>0 \quad \text{ for $x>0$}, \quad \phi_1(0)=0, \quad \lim_{x\to\infty} \phi_1(x)=+\infty.
\end{equation}
\end{lemma}
\begin{proof}
Since $\phi$ is continuous, we see that $\phi_0$ defined in \eqref{def.phi0} is continuous on $[0,\infty)$. 
Moreover, $\phi_0(0)=0$ and $\phi_0(x)>0$ for all $x>0$. It is easy to see that $\phi_1(x)\geq 0$ for all $x\geq 0$. Also, if $\phi_1(x)=0$ for some $x>0$, 
it follows by the non--negativity and continuity of $\phi_0$ that $(v\wedge 1)\phi_0(v)=0$ for a.a. $v\in [x,2x]$. Therefore, it must follow that 
 $\phi_0(v)=0$ for a.a. $v\in [x,2x]$, which is false as $\phi_0(v)>0$ for all $v>0$. Therefore we have $\phi_1(x)>0$ for all $x>0$. 
 \eqref{eq.fasybounded} implies that $\phi(x)\to \infty$ as $x\to\infty$. Therefore it follows that $\phi_0(x)\to\infty$ as $x\to\infty$. Hence for $x>1$ 
 we have
 \[
 \phi_1(x)=\frac{1}{x}\int_x^{2x} \phi_0(v)\,dv,
 \]
 and so it follows that $\phi_1(x)\to\infty$ as $x\to\infty$. By definition, $\phi_1(0)=0$, so all the statements in \eqref{eq.phi1pos} have been verified.
 
Next we show that $\phi_1$ is continuously differentiable on $(0,\infty)$ 
and that $\phi_1'(0+)=0$. This will guarantee that $\phi_1$ is locally Lipschitz continuous. We start by considering the one--sided derivative at $0$. 
Let $x\in (0,1/2]$. Then by \eqref{def.phi1}, we have 
\[
0<\frac{\phi_1(x)}{x}=\frac{1}{x}\int_x^{2x} \frac{v}{x} \phi_0(v)\,dv
\leq 2\frac{1}{x}\int_x^{2x}\phi_0(v)\,dv.
\]
Since $\phi_0$ is continuous and $\phi_0(x)\to 0$ as $x\to 0^+$, we have that the right most member of the above inequality has an indeterminate form 
as $x\to 0$. The continuity of $\phi_0$ allows us to employ l'H\^opital's rule to obtain
\[
\lim_{x\to 0^+} \frac{1}{x}\int_x^{2x}\phi_0(v)\,dv = \lim_{x\to 0^+} \{2\phi_0(2x)-\phi_0(x)\}=0.
\] 
Therefore we have that $\phi_1(x)/x\to 0$ as $x\to 0^+$. Since $\phi_1(0)=0$, it follows that $\phi_1'(0+)=0$. 
For $x>0$, the continuity of $v\mapsto (v\wedge 1)\phi_0(v)$ ensures that $\phi_1'(x)$ is well defined and is given by 
\[
\phi_1'(x)=\frac{1}{x^2}\left(x  \{ 2((2x)\wedge 1) \phi_0(2x) - (x\wedge 1) \phi_0(x)  \} -\int_x^{2x} (v\wedge 1) \phi_0(v)\,dv \right).
\]
We notice also that $\phi_1'$ is continuous on $[0,\infty)$ by the continuity of $\phi_0$ and the fact that for $0<x\leq 1/2$ we have
\[
\phi_1'(x)
=4\phi_0(2x) -  \phi_0(x) -\frac{1}{x^2}\int_x^{2x} v \phi_0(v)\,dv,
\]
so $\lim_{x\to 0^+}\phi_1'(x)=0=\phi_1'(0+)$. 

It remains to prove \eqref{eq.phi1ltphi}. Since $v\wedge 1\leq 1$, by \eqref{def.phi0}, we have for $x>0$ that 
\[
\phi_1(x)= \frac{1}{x}\int_x^{2x} (v\wedge 1) \phi_0(v)\,dv \leq \frac{1}{x}\int_x^{2x} \inf_{v/2\leq y\leq 4v} \phi(y)\,dv.
\]
For $v\in [x,2x]$, it follows that $v/2\leq x$ and that $4v\geq 4x>2x$. Therefore $[v/2,4v]\supset [x,2x]$ for $v\in [x,2x]$. Hence
\[
\inf_{v/2\leq y\leq 4v} \phi(y) \leq \inf_{x\leq y\leq 2x} \phi(y),
\]
and so
\[
\phi_1(x)\leq \frac{1}{x}\int_x^{2x} \inf_{x\leq y\leq 2x} \phi(y)\,dv
\leq \inf_{x\leq y\leq 2x} \phi(y) \leq \phi(x),
\]
as required.
\end{proof}
In our next result, we show that if $\phi_1$ is defined by \eqref{def.phi1} 
the function $\phi_2$ defined by 
\begin{equation} \label{def.phi2}
\phi_2(x):=\sqrt{x}\phi_1(\sqrt{x}), \quad x\geq 0
\end{equation}
is also locally Lipschitz continuous on $[0,\infty)$. This function also plays a role in our comparison proof, and in order to 
apply a standard approach in that proof, we find it convenient that $\phi_2$ be locally Lipschitz continuous.
\begin{lemma} \label{lemma.phi2loclip}
Suppose that $\phi_1$ is locally Lipschitz continuous on $[0,\infty)$, $\phi_1(0)=0$ and $\phi_1(x)>0$ for all $x>0$. If $\phi_2$ is defined by 
\eqref{def.phi2}, then $\phi_2:[0,\infty)\to\mathbb{R}$ is locally Lipschitz continuous.
\end{lemma}
\begin{proof}
Since $\phi_1$ is locally Lipschitz continuous, it follows that for every $n\in \mathbb{N}$ there exists $K_n>0$ such that
\[
|\phi_1(x)-\phi_1(y)|\leq K_n|y-x|, \text{ for all $x, y\in [0,n]$}.
\]
Since $\phi_1(0)=0$, we have that $|\phi_1(x)|\leq K_n x$ for all $x\in [0,n]$.
To prove that $\phi_2$ is locally Lipschitz continuous, suppose that
$x,y\in [0,n]$ and suppose without loss of generality that $0\leq
y\leq x\leq n$. Hence $0\leq \sqrt{y}\leq \sqrt{x}\leq \sqrt{n}$.
Write
\[
\phi_2(x)-\phi_2(y)
=\sqrt{x}(\phi(\sqrt{x})-\phi_1(\sqrt{y}))+\phi_1(\sqrt{y})(\sqrt{x}-\sqrt{y}),
\]
so because $\phi_1$ is non--negative and $\sqrt{x}\geq \sqrt{y}$ we
have
\[
|\phi_2(x)-\phi_2(y)| \leq
\sqrt{x}|\phi_1(\sqrt{x})-\phi_1(\sqrt{y})|+\phi(\sqrt{y})(\sqrt{x}-\sqrt{y}).
\]
Therefore, using the Lipschitz continuity of $\phi_1$ and the estimate
$|\phi(y)|\leq K_{\sqrt{n}}\sqrt{y}$ for all $y\leq n$ we have
\begin{align*}
|\phi_2(x)-\phi_2(y)|
&\leq \sqrt{x} K_{\sqrt{n}}|\sqrt{x}-\sqrt{y}|+K_{\sqrt{n}}\sqrt{y} (\sqrt{x}-\sqrt{y})\\
&=\sqrt{x} K_{\sqrt{n}}(\sqrt{x}-\sqrt{y})+K_{\sqrt{n}}\sqrt{y}
(\sqrt{x}-\sqrt{y}) =K_{\sqrt{n}}(x-y),
\end{align*}
so that $|\phi_2(x)-\phi_2(y)|\leq K_{\sqrt{n}}|x-y|$ for $0\leq
y\leq x\leq n$. Hence $\phi_2$ is locally Lipschitz continuous.
\end{proof}
Let $X$ be a continuous adapted process which obeys \eqref{eq.mainstocheqn}. 
Associated with this solution of \eqref{eq.mainstocheqn}, define the $r$ scalar processes $\bar{\sigma}_j:[0,\infty)\to
\mathbb{R}$ by
\begin{equation} \label{def.sigbarj}
\bar{\sigma}_j(t)= \left\{
\begin{array}{cc}
\sum_{i=1}^d \frac{\langle
X(t),\mathbf{e}_i\rangle}{\|X(t)\|}\sigma_{ij}(t), & X(t)\neq 0, \\
\frac{1}{\sqrt{d}}\sum_{i=1}^d |\sigma_{ij}(t)|, & X(t)=0.
\end{array}
\right.
\end{equation}
We define $\bar{\sigma}(t)\geq 0$ by
\begin{equation} \label{def.barsigma}
\bar{\sigma}^2(t):=\sum_{j=1}^r \bar{\sigma}_j^2(t), \quad t\geq 0.
\end{equation}
Hence $\bar{\sigma}_j$ for $j=1,\ldots,r$ and $\bar{\sigma}$ are
adapted  processes. Therefore using the Cauchy--Schwartz inequality
and \eqref{def.sigbarj} we get
\[
\bar{\sigma}_j^2(t)\leq \sum_{i=1}^d \sigma_{ij}^2(t), \quad t\geq
0,
\]
and so $\bar{\sigma}^2(t)\leq \|\sigma(t)\|^2_F$ for all $t\geq 0$.
Hence $\bar{\sigma}$ and $\bar{\sigma}_j$ for $j=1,\ldots,r$ are
bounded functions on any compact interval. Therefore, the scalar process
$\tilde{Y}_0$ given by
\[
\tilde{Y}_0(t)=\sum_{j=1}^r \int_0^t e^s \bar{\sigma}_j(s)\,dB_j(s),
\quad t\geq 0
\]
is well--defined and is moreover a continuous square integrable
martingale.
Therefore the process $Y_0$ defined by
\begin{equation}\label{def.Y0}
Y_0(t)=e^{-t}\tilde{Y}_0(t), \quad  t\geq 0
\end{equation}
is a continuous semimartingale and obeys
\begin{equation} \label{eq.Y0sde}
dY_0(t)=-Y_0(t)\,dt + \sum_{j=1}^r \bar{\sigma}_j(t)\,dB_j(t), \quad
t\geq 0.
\end{equation}
Next define $W(0)=1+\|\xi\|>0$ and
\begin{equation} \label{def.WcompSDE}
W'(t)=
-\phi_1(W(t)+Y_0(t))+\frac{\|\sigma(t)\|^2_F+e^{-t}}{W(t)+Y_0(t)} +
Y_0(t), \quad t\geq 0,
\end{equation}
where $\phi_1$ is defined by \eqref{def.phi1}. By Lemma~\ref{lemma.phi1loclip}, $\phi_1$ is locally Lipschitz continuous; also, $\|\sigma\|^2_F$ is
continuous and the paths of $Y_0$ are continuous, so there is a unique
continuous solution of \eqref{def.WcompSDE} on the interval
$[0,\tau)$ where
\begin{equation}  \label{def.tauZ}
\tau=\inf\{t>0: Z(t)\not\in (0,\infty)\}
\end{equation}
and 
\begin{equation} \label{def.Z}
Z(t)=W(t)+Y_0(t), \quad \text{for $t\in [0,\tau)$}. 
\end{equation}
We understand that when we speak of a \emph{unique} solution of \eqref{def.WcompSDE}, we mean that it is a unique 
solution corresponding to a \emph{given} solution $X$ of \eqref{eq.mainstocheqn}. Of course, as our continuity assumption on $f$ may be too weak to ensure 
that there is a unique solution  $X$ of \eqref{eq.mainstocheqn}, we do not expect there to be unique solutions of \eqref{def.WcompSDE}, but merely unique 
relative to a given solution $X$ of \eqref{eq.mainstocheqn}.

Therefore, as $W$ is the
unique continuous solution of \eqref{def.WcompSDE} on $[0,\tau)$ for a given $X$, it
follows that on $[0,\tau)$ that $Z$ defined in \eqref{def.Z} is the unique
solution of the stochastic differential equation
\begin{equation} \label{def.ZcompSDE}
dZ(t)=\left(-\phi_1(Z(t))+\frac{\|\sigma(t)\|^2_F+e^{-t}}{Z(t)}\right)\,dt
+ \sum_{j=1}^r \bar{\sigma}_j(t)\,dB_j(t),
\end{equation}
for a given $X$, with initial condition $Z(0)=\|\xi\|+1>0$. The adaptedness of $Y_0$
ensures that the process $W$ is adapted, and therefore so is $Z$.

The first step is to show that $\tau=+\infty$ a.s., which means that
$Z(t)$ is well--defined and strictly positive for all $t\geq 0$,
a.s. In the rest of this section, when we say that certain processes are ``unique'' solutions 
of certain stochastic differential equations, we mean that the process is unique given a specific 
solution $X$ of \eqref{eq.mainstocheqn}.  
\begin{lemma} \label{lemma.Zpos}
Suppose that $f$ obeys \eqref{eq.fglobalunperturbed}, 
and that $\sigma$ obeys \eqref{eq.sigmacns}. Let
$Z$ be the unique continuous adapted solution of
\eqref{def.ZcompSDE}. Then $\tau$ defined by \eqref{def.tauZ} is such that 
$\tau=+\infty$ a.s.
\end{lemma}
\begin{proof}
Let $\zeta=\|\xi\|+1>0$ and define $k^\ast\in\mathbb{N}$ such that
$k^\ast>\zeta$. Define for each $k\geq k^\ast$ the stopping time
$\tau_k^\zeta=\inf\{t>0:Z(t)=k \text{ or } 1/k\}$. We see that
$\tau_k^\zeta$ is an increasing sequence of times and so
$\tau_\infty^\zeta:=\lim_{k\to\infty} \tau_k^\zeta$. Suppose, in
contradiction to the desired claim,
 that $\tau_\infty^\zeta<+\infty$ with positive probability for some $\zeta$. Then, there exists $T>0$, $\epsilon>0$ and $k_0\in\mathbb{N}$ such that
\[
\mathbb{P}[\tau_k^\zeta\leq T] \geq \epsilon, \quad k\geq
k_0>k^\ast.
\]
Therefore, by It\^o's rule we have that
\begin{multline*}
Z(T\wedge \tau_k^\zeta)+\frac{1}{Z(T\wedge \tau_k^\zeta)}=
\zeta+\frac{1}{\zeta}\\+\int_0^{T\wedge \tau_k^\zeta}
\left\{-\phi_1(Z(s))+\frac{\phi_1(Z(s))}{Z(s)}\frac{1}{Z(s)}-\frac{e^{-s}}{Z(s)^3}
+\frac{\|\sigma(s)\|^2_F+e^{-s}}{Z(s)}  \right\}\,ds
\\
+\sum_{j=1}^r \int_0^{T\wedge \tau_k^\zeta}
(1-Z(s)^{-2})\bar{\sigma}_j(s)\,dB_j(s).
\end{multline*}
We remove the non--autonomous terms in the first integral by noting
that $\|\sigma(s)\|^2_F\leq \sigma_T^2<+\infty$ for all $s\in
[0,T]$, so we arrive at
\begin{equation*}
Z(T\wedge \tau_k^\zeta)+\frac{1}{Z(T\wedge \tau_k^\zeta)}=
\zeta+\frac{1}{\zeta}+\int_0^{T\wedge \tau_k^\zeta} b_T(Z(s))\,ds
+ M(T)
\end{equation*}
where we have defined
\begin{equation} \label{eq.bT}
b_T(z)= -\phi_1(z)+\frac{\phi_1(z)}{z}\frac{1}{z}-\frac{e^{-T}}{z^3}
+\frac{1+\sigma_T^2}{z}, \quad z>0,
\end{equation}
and $M=\{M(t):t\in [0,T]\}$ is the martingale defined by
\[
M(t)=\sum_{j=1}^r \int_0^{t\wedge \tau_k^\zeta}
(1-Z(s)^{-2})\bar{\sigma}_j(s)\,dB_j(s), \quad t\in [0,T].
\]
For $z\geq 1$, since $\phi_1(z)\geq 0$ for all $z\geq 0$, we have
\[
b_T(z)= -\phi_1(z)(1-z^{-2}) -\frac{e^{-T}}{z^3}
+\frac{1+\sigma_T^2}{z}\leq \frac{1+\sigma_T^2}{z}\leq 1+\sigma_T^2.
\]
For $z\in (0,1]$, the Lipschitz continuity of $\phi_1$ and the fact that $\phi_1(0)=0$ 
guarantees that $|\phi_1(z)|\leq K_1 z$ for some $K_1>0$. Therefore we have
\[
b_T(z)\leq \frac{K_1+1+\sigma_T^2}{z}-\frac{e^{-T}}{z^3},
\]
and so we can readily show that there is $K_2(T)>0$ such that
$b_T(z)\leq K_2(T)$ for all $z\in (0,1]$. Define
$K_3(T)=\max(K_2(T),1+\sigma^2_T)$. Therefore we have $b_T(z)\leq
K_3(T)$ for all $z>0$. Since $Z(s)\in (0,\infty)$ for all $s\in
[0,T\wedge \tau_k^\zeta]$ we have that
\[
Z(T\wedge \tau_k^\zeta)+\frac{1}{Z(T\wedge \tau_k^\zeta)}\leq
\zeta+\frac{1}{\zeta}+\int_0^{T\wedge \tau_k^\zeta} K_3(T)+ M(T)
\leq \zeta+\frac{1}{\zeta}+ TK_3(T)+ M(T).
\]
By the optional sampling theorem, we have that
\[
\mathbb{E}\left[Z(T\wedge \tau_k^\zeta)+\frac{1}{Z(T\wedge
\tau_k^\zeta)}\right]\leq
 \zeta+\frac{1}{\zeta}+ TK_3(T)=:K(T,\zeta)<+\infty.
\]
Define next the event $C_k=\{\tau_k^\zeta\leq T\}$. Then for $k\geq
k_0$ we have $\mathbb{P}[C_k]\geq \epsilon$. If $\omega\in C_k$, we
have that $\tau_k^\zeta\leq T$, so $Z(T\wedge \tau_k^\zeta)=k$ or
$Z(T\wedge \tau_k^\zeta)=1/k$. Hence $Z(T\wedge
\tau_k^\zeta)+1/Z(T\wedge \tau_k^\zeta)=k+1/k$ for $\omega\in C_k$.
Hence
\begin{align*}
K(T,\zeta)&\geq \mathbb{E}\left[Z(T\wedge \tau_k^\zeta)+\frac{1}{Z(T\wedge \tau_k^\zeta)}\right]\\
&\geq \mathbb{E}\left[\left(Z(T\wedge \tau_k^\zeta)+\frac{1}{Z(T\wedge \tau_k^\zeta)}\right)I_{C_k}\right]\\
&=(k+1/k)\mathbb{P}[C_k]\geq (k+1/k)\epsilon.
\end{align*}
Therefore, we have that $K(T,\zeta)\geq (k+1/k)\epsilon$ for all
$k\geq k_0$. Letting $k\to\infty$ gives a contradiction.
\end{proof}
Given that $Z$ is positive and well--defined for all $t\geq 0$, we
are now in a position to formulate and prove a comparison result,
which shows that $\|X(t)\|\leq Z(t)$ for all $t\geq 0$ a.s. Once
this result is proven, the main theorem will be established if we
show that the solution $Z$ of \eqref{def.ZcompSDE} is bounded.
\begin{lemma} \label{lemma.XZcomp}
Suppose that $f$ obeys \eqref{eq.fglobalunperturbed} 
and that $\sigma$ obeys \eqref{eq.sigmacns}.
Suppose that $X$ is a continuous adapted process which obeys \eqref{eq.mainstocheqn}, 
and let $Z$ be the unique continuous adapted process  corresponding to $X$ which obeys \eqref{def.ZcompSDE}.
Then $\|X(t)\|\leq Z(t)$ for all $t\geq 0$ a.s.
\end{lemma}
\begin{proof}
Define $Y_2(t)=\|X(t)\|^2$ for $t\geq 0$. Then by the definition of
$\bar{\sigma}_j$ for $j=1,\ldots,r$ from \eqref{def.sigbarj}, we
have
\[
2\sum_{i=1}^d X_i(t)\sigma_{ij}(t)=2\sqrt{Y_2(t)}\bar{\sigma}_j(t),
\quad t\geq 0.
\]
By It\^o's rule, we have
\[
dY_2(t)=\left( -2\langle X(t),f(X(t))\rangle
+\|\sigma(t)\|^2_F\right)\,dt + 2 \sum_{j=1}^r \sum_{i=1}^d
X_i(t)\sigma_{ij}(t)\,dB_j(t), \quad t\geq 0.
\]
Using this semimartingale decomposition and the previous identity,
we get
\begin{equation} \label{eq.Y2diffok}
dY_2(t)=\left( -2\langle X(t),f(X(t))\rangle
+\|\sigma(t)\|^2_F\right)\,dt + 2\sqrt{Y_2(t)} \sum_{j=1}^r
\bar{\sigma}_j(t)\,dB_j(t).
\end{equation}
Let $\phi_1$ be the function defined by \eqref{def.phi1}, 
$\bar{\sigma}$ the process defined by \eqref{def.barsigma}, and
define the processes $\eta_1$ and $\eta_2$ by
\begin{align*}
\eta_1(t)&= \|\sigma(t)\|^2_F + 2e^{-t} +\bar{\sigma}(t)^2, \quad t\geq 0,\\
\eta_2(t)&=2\sqrt{Y_2(t)}\phi_1(\sqrt{Y_2(t)})-2\langle
X(t),f(X(t))\rangle, \quad t\geq 0,
\end{align*}
and the processes $\beta_1$ and $\beta_2$ by
\begin{align}
\label{def.beta1}
\beta_1(t)&=b(Z_2(t),t) +\eta_1(t), \quad t\geq 0,\\
\label{def.beta2} \beta_2(t)&=b(Y_2(t),t) +\eta_2(t), \quad t\geq 0,
\end{align}
where we have defined $b:[0,\infty)\times [0,\infty)\to\mathbb{R}$
by
\begin{equation} \label{def.b}
b(x,t)=-2\phi_2(x)+\|\sigma(t)\|^2_F, \quad x\geq 0, t\geq 0,
\end{equation}
where $\phi_2$ is defined in \eqref{def.phi2}. 

Granted these definitions, we can rewrite \eqref{eq.Y2diffok} as
\begin{equation} \label{eq.Y2diffcomp}
dY_2(t)=
\beta_2(t)\,dt + 2\sqrt{Y_2(t)} \sum_{j=1}^r
\bar{\sigma}_j(t)\,dB_j(t).
\end{equation}
Next, by virtue of Lemma~\ref{lemma.Zpos} it follows that there is a
positive  process $Z_2=\{Z_2(t):t\geq 0\}$ defined by
$Z_2(t)=Z(t)^2$ for all $t\geq 0$. Therefore, applying It\^o's rule
to \eqref{def.ZcompSDE}, and using the definition
\eqref{def.barsigma}, we have
\begin{multline*}
dZ_2(t)=\left(2Z(t)\left\{-\phi_1(Z(t))+\frac{e^{-t}+\|\sigma(t)\|^2_F}{Z(t)}
\right\}+\bar{\sigma}^2(t)  \right)\,dt
\\+ 2Z(t)\sum_{j=1}^r \bar{\sigma}_j(t)\,dB_j(t).
\end{multline*}
Hence by the definition of $\phi_2$, \eqref{def.beta1} and $Z_2$ we
have
\begin{equation} \label{eq.Z2diffcomp}
dZ_2(t)=
\beta_1(t)\,dt + 2\sqrt{Z_2(t)} \sum_{j=1}^r
\bar{\sigma}_j(t)\,dB_j(t).
\end{equation}
Notice also that $Y_2(0)=\|\xi\|^2<1+\|\xi\|^2=Z_2(0)$.

Our proof now involves comparing $Y_2$ and $Z_2$, viewed as
solutions of \eqref{eq.Y2diffcomp} and \eqref{eq.Z2diffcomp}
respectively. Proving that $Y_2(t)\leq Z_2(t)$ for all $t\geq 0$
a.s. suffices. The proof is an adaptation of standard comparison
proofs. Extant results can not be applied immediately, because we
must carefully deal with the fact that the state--dependence in the
drift in both \eqref{eq.Y2diffcomp} and \eqref{eq.Z2diffcomp} is
merely \emph{locally} Lipschitz continuous, and that the diffusion
coefficients are non--autonomous through the presence of a
\emph{process} rather than simple deterministic dependence of time.

To prove that $Y_2$ is dominated by $Z_2$, we first show that
$\eta_1(t)>0\geq \eta_2(t)$ for $t\geq 0$. The first inequality is
immediate. To show that $\eta_2(t)\leq 0$ for all $t\geq 0$, first
note that if $X(t)=0$, then $\eta_2(t)=0$. If $\|X(t)\|>0$, by
\eqref{def.phimult} and the definition of $Y_2$, we have that
\[
\frac{\langle X(t),f(X(t))\rangle}{\|X(t)\|}\geq
\inf_{\|x\|=\|X(t)\|} \frac{\langle
x,f(x)\rangle}{\|x\|}=\phi(\sqrt{Y_2(t)}).
\]
Next, if $\phi_1$  is the function defined in \eqref{def.phi1}, by \eqref{eq.phi1ltphi} we have 
\[
\frac{\langle X(t),f(X(t))\rangle}{\|X(t)\|}\geq \phi(\sqrt{Y_2(t)}) \geq \phi_1(\sqrt{Y_2(t)}).
\]
Hence $\langle X(t),f(X(t))\rangle\geq
\|X(t)\|\phi_1(\sqrt{Y_2(t)})=\sqrt{Y_2(t)}\phi_1(\sqrt{Y_2(t)})$, so
$\eta_2(t)\leq 0$. Therefore, because $\eta_2\leq 0$ and $\eta_1>0$,
we have
\begin{equation} \label{eq.bbeta12}
\beta_2(t)\leq b(Y_2(t),t), \quad \beta_1(t)> b(Z_2(t),t), \quad
t\geq 0.
\end{equation}
By Lemma~\ref{lemma.phi2loclip}, $\phi_2$ is locally Lipschitz
continuous, so for every $n\geq 0$ there is a $\kappa_n>0$ such that
\begin{equation} \label{eq.btlip}
|b(x,t)-b(y,t)|=|2\phi_2(x)-2\phi_2(y)|\leq \kappa_n|x-y|
\quad\text{for all $x,y\in [0,n]$}.
\end{equation}
Now define $\Delta(t):=Y_2(t)-Z_2(t)$ for $t\geq 0$. Let
$\rho(x)=4x$ for $x\geq 0$. Then $\rho$ is increasing and
$\int_{0^+} 1/\rho(x)\,dx = +\infty$. Now by \eqref{def.barsigma}
\[
d[\Delta](t)=4\left(\sqrt{Y_2(t)}-\sqrt{Z_2(t)}\right)^2
\sum_{j=1}^r \bar{\sigma}_j^2(t)\,dt
=4\left(\sqrt{Y_2(t)}-\sqrt{Z_2(t)}\right)^2 \bar{\sigma}^2(t)\,dt.
\]
If
\begin{equation} \label{eq.loctimeint}
\int_0^t
\rho(\Delta(s))^{-1}I_{\{\Delta(s)>0\}}\,d[\Delta](s)<+\infty,
\text{a.s.}
\end{equation}
then $\Lambda_t^0(\Delta)=0$ a.s., where $\Lambda_\cdot^0(\Delta)$
is the local time of $\Delta$ in zero (see \cite[Proposition
V.39.3]{RogWill:1989}).

If $y\geq x\geq 0$, we have that $(\sqrt{y}-\sqrt{x})^2\leq y-x$.
Define $J=\{s\in [0,t]: \Delta(s)>0\}$. Therefore, $s\in J$ we have
$Y_2(s)>Z_2(s)>0$ and so
\[
\left(2\sqrt{Y_2(t)}-2\sqrt{Z_2(t)}\right)^2\leq
4(Y_2(s)-Z_2(s))=4\Delta(s)=\rho(\Delta(s)).
\]
Thus
\begin{align*}
\lefteqn{\int_0^t \rho(\Delta(s))^{-1}I_{\{\Delta(s)>0\}}\,d[\Delta](s)}\\
&=\int_{J} \rho(\Delta(s))^{-1}I_{\{\Delta(s)>0\}}\,d[\Delta](s)
+\int_{[0,t]\setminus J} \rho(\Delta(s))^{-1}I_{\{\Delta(s)>0\}}\,d[\Delta](s)\\
&=\int_{J} \rho(\Delta(s))^{-1} \cdot 4\left(\sqrt{Y_2(s)}-\sqrt{Z_2(s)}\right)^2 \bar{\sigma}^2(s)\,ds\\
&\leq \int_{J}  \bar{\sigma}^2(s)\,ds\leq \int_0^t
\bar{\sigma}^2(s)\,ds \leq \int_0^t \|\sigma(s)\|^2_F\,ds<+\infty,
\end{align*}
as required.

Next, let
\[
\tau_n=\inf\{t>0:Y_2(t)=n \text{ or } Z_2(t)=n\}, \quad n\geq
\lceil 1+\|\xi\|^2\rceil.
\]
By Lemma~\ref{lemma.Zpos}, $Z$ does not explode in finite time, so
neither does $Z_2$. Also, as $\|X\|$ does not explode in finite
time, we have that $\tau_n\to\infty$ as $n\to\infty$. Using the fact
that $\Lambda_t^0(\Delta)=0$ a.s., together with
\eqref{eq.Y2diffcomp} and \eqref{eq.Z2diffcomp} we get
\begin{equation} \label{eq.Delpsemimart1}
\Delta(t\wedge \tau_n)^+=\Delta(0)^+ + \int_0^{t\wedge \tau_n}
I_{\{\Delta(s)>0\}}(\beta_2(s)-\beta_1(s))\,ds + M(t),
\end{equation}
where we have defined the local martingale $M$ by
\[
M(t)=\int_0^{t\wedge \tau_n} I_{\{\Delta(s)>0\}}2
\left(\sqrt{Y_2(s)}-\sqrt{Z_2(s)}\right)\sum_{j=1}^r
\bar{\sigma}_j(s)\,dB_j(s).
\]
Therefore by \eqref{def.barsigma}, and the fact that
$\sqrt{Y_2(s)}\vee \sqrt{Z_2(s)}\leq \sqrt{n}$ for $s\in [0,t\wedge
\tau_n]$
\begin{align*}
\langle M\rangle(t)
&= 4\int_0^{t\wedge \tau_n} I_{\{\Delta(s)>0\}} \left(\sqrt{Y_2(s)}-\sqrt{Z_2(s)}\right)^2  \bar{\sigma}^2(s)\,ds\\
&\leq 4\int_0^{t\wedge \tau_n} I_{\{\Delta(s)>0\}} \left(\sqrt{Y_2(s)}-\sqrt{Z_2(s)}\right)^2 \|\sigma(s)\|^2_F\,ds\\
&\leq 4n\int_0^{t\wedge \tau_n}  \|\sigma(s)\|^2_F\,ds\leq
4n\int_0^t  \|\sigma(s)\|^2_F\,ds.
\end{align*}
Now $\Delta(0)=Y_2(0)-Z_2(0)<0$, so by the optional sampling
theorem, we deduce from \eqref{eq.Delpsemimart1} that
\begin{equation} \label{eq.EDelpsemimart1}
0\leq \mathbb{E}[\Delta(t\wedge
\tau_n)^+]=\mathbb{E}\left[\int_0^{t\wedge \tau_n}
I_{\{\Delta(s)>0\}}(\beta_2(s)-\beta_1(s))\,ds\right].
\end{equation}

We now estimate the integrand on the right--hand side. If
$\Delta(s)>0$, we have $\Delta(s)=Y_2(s)-Z_2(s)>0$. Thus for $s\in
[0,t\wedge \tau_n]$, because $Y_2(s)\vee Z_2(s)\leq n$, we may use
\eqref{eq.bbeta12} and then \eqref{eq.btlip} to get
\begin{align*}
I_{\{\Delta(s)>0\}}(\beta_2(s)-\beta_1(s))
&=\beta_2(s)-\beta_1(s) \leq b(Y_2(s),s)-b(Z_2(s),s)\\
&\leq |b(Y_2(s),s)-b(Z_2(s),s)| \leq \kappa_n|Y_2(s)-Z_2(s)|.
\end{align*}
Since $Y_2(s)-Z_2(s)>0$, this gives
$I_{\{\Delta(s)>0\}}(\beta_2(s)-\beta_1(s))\leq \kappa_n
(Y_2(s)-Z_2(s))=\kappa_n \Delta(s)^+$. In the case when
$\Delta(s)\leq 0$, we have
$I_{\{\Delta(s)>0\}}(\beta_2(s)-\beta_1(s))=0\leq \kappa_n
\Delta(s)^+$. Thus, the estimate
$I_{\{\Delta(s)>0\}}(\beta_2(s)-\beta_1(s))=0\leq \kappa_n
\Delta(s)^+$ holds for all $s\in [0,t\wedge\tau_n]$, so inserting
this bound into \eqref{eq.EDelpsemimart1}, we get
\begin{equation}\label{eq.EDelpsemimart2}
0\leq \mathbb{E}[\Delta(t\wedge \tau_n)^+] \leq
\mathbb{E}\left[\int_0^{t\wedge \tau_n}  \kappa_n
\Delta(s)^+\,ds\right]=\kappa_n\mathbb{E}\int_0^{t\wedge \tau_n}
\Delta(s)^+\,ds.
\end{equation}
As to the term on the righthand side, by considering the cases when
(a) $\tau_n\leq t$ and (b) $\tau_n>t$, we can show that
\[
\int_0^{t\wedge \tau_n} \Delta(s)^+\,ds \leq \int_0^t \Delta(s\wedge
\tau_n)^+\,ds.
\]
Putting this estimate into \eqref{eq.EDelpsemimart2} gives
\begin{equation}\label{eq.EDelpsemimart3}
0\leq \mathbb{E}[\Delta(t\wedge \tau_n)^+] \leq \kappa_n \int_0^{t}
\mathbb{E}[\Delta(s\wedge \tau_n)^+]\,ds, \quad t\geq 0.
\end{equation}
Since $t\mapsto \Delta(t)$ has a.s. continuous sample paths, so does
$t\mapsto \Delta(t\wedge \tau_n)$, and therefore
$\delta_n:[0,\infty)\to \mathbb{R}$ defined by
$\delta_n(t)=\mathbb{E}[\Delta(t\wedge \tau_n)]$ for $t\geq 0$ is a
non--negative and continuous function obeying $\delta_n(t)\leq
\kappa_n\int_0^t \delta_n(s)\,ds$ for all $t\geq 0$. By Gronwall's
inequality, $\delta_n(t)=0$ for all $t\geq 0$. Therefore we have
$Y_2(t\wedge \tau_n)-Z_2(t\wedge \tau_n)\leq 0$ for all $t\geq 0$
a.s. and for each $n\in\mathbb{N}$. Since $\tau_n\to\infty$ as
$n\to\infty$, it follows that $Y_2(t)-Z_2(t)\leq 0$ for all $t\geq
0$ a.s., as required.
\end{proof}
In the next lemma, we show that $Y_0$ defined by \eqref{def.Y0} is
bounded. We notice that the bound on the solution is deterministic, and therefore does not depend on 
the process $X$, which is a solution of \eqref{eq.mainstocheqn}, and on which $Y_0$ depends. 
\begin{lemma} \label{lemma.Y0bounded}
Suppose that $S_h'$ obeys \eqref{eq.thetaboundedh}. If $Y_0$ is defined
by \eqref{def.Y0}, then there is $c_1>0$ such that
\[
\limsup_{t\to\infty} |Y_0(t)|\leq c_1, \quad \text{a.s.}
\]
\end{lemma}
\begin{proof}
We start the proof by showing that we may consider $h=1$ without loss of generality. 
If $S_h'$  obeys \eqref{eq.thetaboundedh}, it follows that $S_1'$ also obeys \eqref{eq.thetaboundedh}, in the sense that
there exists $\epsilon'>0$ such that
\begin{equation}  \label{eq.S1mixed}
S_1'(\epsilon)<+\infty \text{ for all $\epsilon>\epsilon'$ and } S_1'(\epsilon)=+\infty \text{ for all $\epsilon<\epsilon'$}
\end{equation}
Suppose that \eqref{eq.S1mixed} is not true. Then either $S_1'(\epsilon)=+\infty$ for all $\epsilon>0$ or $S_1'(\epsilon)<+\infty$ for all $\epsilon>0$.
The fact that $S_h'$  obeys \eqref{eq.thetaboundedh} implies from Theorem~\ref{theorem.Yclassify} that the process $Y$ defined by \eqref{def.Y}
obeys
\[
0<c_1'\limsup_{t\to\infty} \|Y(t)\|\leq c_2', \quad\text{a.s.}
\]
for some positive deterministic constants $c_1'$ and $c_2'$. 
If $S_1'(\epsilon)=+\infty$ for all $\epsilon>0$, then by part (C) of Theorem~\ref{theorem.Yclassify} we have that $\limsup_{t\to\infty}\|Y(t)\|=+\infty$ a.s.
a contradiction. On the other hand, if $S_1'(\epsilon)<+\infty$ for all $\epsilon>0$, by part (A) of 
Theorem~\ref{theorem.Yclassify} we have that $\lim_{t\to\infty} Y(t)=0$ a.s., which is also a contradiction. Therefore, it must be that \eqref{eq.S1mixed}
holds. Notice also that \eqref{eq.S1mixed} implies 
\begin{equation} \label{eq.intsignto0}
\lim_{n\to\infty}\int_{n-1}^n \|\sigma(s)\|^2_F\,ds =0.
\end{equation}

We now start the proof in earnest. Let $V_0(n):=\int_{n-1}^n e^{s-n}\sum_{j=1}^r
\bar{\sigma}_j(s)\,dB(s)$, $n\geq 1$. Then by  \eqref{def.Y0} we get
\begin{equation} \label{eq.YV0convrep}
Y_0(n)=e^{-n}\sum_{l=1}^n \int_{l-1}^l e^s\sum_{j=1}^r
\bar{\sigma}_j(s)\,dB_j(s)=\sum_{l=1}^n e^{-(n-l)}V_0(l), \quad
n\geq 1.
\end{equation}

Define
\[
\tilde{Y}_{n-1}(t)=\int_{n-1}^t e^s \sum_{j=1}^r
\bar{\sigma}_j(s)\,dB_j(s), \quad t\in [n-1,n].
\]
Clearly $\tilde{Y}_{n-1}$ is a continuous $\mathcal{F}^B$
martingale, and by \eqref{def.barsigma} we have
\[
\langle \tilde{Y}_{n-1}\rangle (t)=\int_{n-1}^t e^{2s}
\bar{\sigma}^2(s)\,ds, \quad t\in [n-1,n].
\]
Therefore there is an extension
$(\Omega_{n},\mathcal{F}_n,\mathbb{P}_n)$ of
$(\Omega,\mathcal{F},\mathbb{P})$ on which is defined a
one--dimensional Brownian  motion $\bar{B}_n=\{\bar{B}_n(t):n-1\leq
t\leq n;\mathcal{F}_n\}$ such that
\[
\tilde{Y}_{n-1}(t)=\int_{n-1}^t e^s  \bar{\sigma}(s)\,d\bar{B}_n(s),
\quad t\in [n-1,n].
\]
(cf.~\cite[Theorem 3.4.2]{KarShr:90}). Now define
\[
\bar{Y}_{n-1}(t)=\int_{n-1}^t e^s\|\sigma(s)\|_F\,d\bar{B}(s), \quad
t\in [n-1,n].
\]
Since $\bar{\sigma}(t)\leq \|\sigma(t)\|_F$ for all $t\geq 0$, by
applying a result of Hajek (cf.~e.g.,~\cite[Exercise 3.4.24]{KarShr:90})
we have that
\begin{equation} \label{eq.pluseps}
\mathbb{P}[V_0(n)>\epsilon]=\mathbb{P}[\tilde{Y}_{n-1}(n)>\epsilon
e^n]\leq 2 \mathbb{P}[\bar{Y}_{n-1}(n)\geq \epsilon e^n].
\end{equation}
Noting that $-\tilde{Y}_{n-1}$ is also a continuous martingale, by
applying Hajek's result once more, we have that
\[
\mathbb{P}[V_0(n)\leq -\epsilon]=\mathbb{P}[-\tilde{Y}_{n-1}(n)\geq
\epsilon e^n]\leq 2\mathbb{P}[\bar{Y}_{n-1}(n)\geq \epsilon e^n].
\]
Combining this estimate with  \eqref{eq.pluseps}, we get
\begin{equation} \label{eq.upboundV0n}
\mathbb{P}[|V_0(n)|>\epsilon]\leq 4 \mathbb{P}[\bar{Y}_{n-1}(n)\geq
\epsilon e^n].
\end{equation}
Now, we notice that $\bar{Y}_{n-1}(n)$ is a normally distributed
random variable with mean zero and variance
\[
\bar{v}(n)^2:=\int_{n-1}^n e^{2s}\|\sigma(s)\|_F^2\,ds.
\]
Notice that $e^{-2}\theta(n)^2\leq e^{-2n}\bar{v}^2(n)\leq
\theta(n)^2$, where
\[
\theta(n)^2=\int_{n-1}^n \|\sigma(s)\|_F^2\,ds.
\] 
Denote by $\Phi:\mathbb{R}\to \mathbb{R}$ the distribution of a
standard normal random variable i.e.,
\begin{equation} \label{def.Phi}
\Phi(x)=\frac{1}{\sqrt{2\pi}}\int_{-\infty}^x e^{-u^2/2}\,du, \quad x\in\mathbb{R}.
\end{equation}
Since $\Phi$ is increasing, we have
\begin{align*} 
\mathbb{P}[|V_0(n)|>\epsilon]&\leq 4 \left(1-
\Phi\left(\frac{\epsilon e^n}{\bar{v}(n)}\right)\right)
= 4 \left(1- \Phi\left(\frac{\epsilon}{e^{-n}\bar{v}(n)}\right)\right)\\
&\leq 4 \left(1- \Phi\left(\frac{\epsilon}{\theta(n)}\right)\right).
\end{align*}
Therefore, for every $\epsilon>\epsilon'$, by \eqref{eq.S1mixed}, \eqref{eq.intsignto0}
and the asymptotic estimate  
\begin{equation} \label{eq.mills}
\lim_{x\to\infty} \frac{1-\Phi(x)}{\frac{1}{x}e^{-x^2/2}}=\frac{1}{\sqrt{2\pi}},
\end{equation}
(cf.,~ e.g.~\cite[Problem 2.9.22]{KarShr:90})
it follows that
\[
\sum_{n=1}^\infty \mathbb{P}[|V_0(n)|\geq \epsilon]<+\infty.
\]
Thus by the Borel--Cantelli lemma, it follows that
$\limsup_{n\to\infty} |V_0(n)|\leq \epsilon$ a.s. for every
$\epsilon>\epsilon'$. Hence by \eqref{eq.YV0convrep}, we have that
\begin{equation} \label{eq.Y0bound}
\limsup_{n\to\infty} |Y_0(n)|\leq \epsilon\cdot \sum_{k=0}^\infty
e^{-k} = \epsilon \frac{1}{1-e^{-1}}, \quad \text{a.s.}
\end{equation}

Next let $t\in [n,n+1)$. Therefore, from \eqref{def.Y0} we have
\[
Y_0(t)=Y_0(n)e^{-(t-n)}+ e^{-t}\int_n^t e^s\sum_{j=1}^r
\bar{\sigma}_{j}(s)\,dB_j(s), \quad t\in [n,n+1).
\]
With $Z_0(n):=e^{-n}\max_{t\in [n,n+1]}  \left|\int_n^t
e^s\sum_{j=1}^r \bar{\sigma}_j(s)\,dB_j(s)\right|$ for $n\geq 1$, we
have
\begin{equation}\label{eq.supY0Y0nZ0n}
\max_{t\in [n,n+1]} |Y_0(t)|\leq |Y_0(n)|+ \max_{t\in [n,n+1]}
e^{-t}\left|\int_n^t e^s\sum_{j=1}^r
\bar{\sigma}_j(s)\,dB_j(s)\right| \leq |Y_0(n)|+ Z_0(n).
\end{equation}
Next we estimate $\mathbb{P}[Z_0(n)>\epsilon e]$. Fix $n\in
\mathbb{N}$. Now
\[
\mathbb{P}[Z_0(n)>\epsilon e]=\mathbb{P}\left[\max_{t\in [n,n+1]}
|\bar{Y}_n(t)|>\epsilon e e^n\right].
\]
Define $\tau(t):=\int_n^t e^{2s}\bar{\sigma}^2(s)\,ds$ for
$t\in[n,n+1]$. Therefore, by the martingale time change
theorem~\cite[Theorem V.1.6]{RevYor}, there exists a standard
Brownian motion $B^\ast_n$ such that
\begin{align*}
\mathbb{P}[Z_0(n)>\epsilon e] =\mathbb{P}\left[\max_{t\in [n,n+1]}
\left|B^\ast_n\left(\tau(t)\right)\right|>\epsilon e e^n\right]
=\mathbb{P}\left[\max_{u\in [0,\tau(n+1)]}
\left|B^\ast_n(u)\right|>\epsilon  e e^n \right].
\end{align*}
Notice now that $\tau(t)\leq \int_n^t e^{2s}\|\sigma(s)\|^2_F\,ds$,
so
\begin{align*}
\mathbb{P}[Z_0(n)>\epsilon e]
&\leq \mathbb{P}\left[\max_{u\in [0, \int_n^{n+1} e^{2s}\|\sigma(s)\|^2_F\,ds]} \left|B^\ast_n(u)\right|>\epsilon  e e^n \right]\\
&= \mathbb{P}\left[\max_{u\in [0, \bar{v}^2(n+1)]} \left|B^\ast_n(u)\right|>\epsilon  e e^n \right]\\
&\leq \mathbb{P}\left[\max_{u\in [0,\bar{v}^2(n+1)]}
B^\ast_n(u)>\epsilon   e^n e\right]
+\mathbb{P}\left[\max_{u\in [0,\bar{v}^2(n+1)]} -B^\ast_n(u)>\epsilon   e^n e\right]\\
&=\mathbb{P}\left[|B^\ast_n(\bar{v}^2(n+1))|>\epsilon e^n e\right]
+\mathbb{P}\left[ |B^{\ast\ast}_n(\bar{v}^2(n+1))|>\epsilon  e^n
e\right],
\end{align*}
where $B^{\ast\ast}_n=-B^\ast_n$ is a standard Brownian motion, and we have recalled that if $W$ is a 
standard Brownian motion that $\max_{s\in [0,t]} W(s)$ has the same distribution as $|W(t)|$. Therefore, as
$B^\ast_n(\bar{v}(n+1))$ is normally distributed with zero mean we
have
\begin{align*}
\mathbb{P}[Z_0(n)>\epsilon e]
&=2\mathbb{P}\left[|B^\ast_n(\bar{v}^2(n+1))|>\epsilon e e^n\right]
=4\mathbb{P}\left[B^\ast_n(\bar{v}^2(n+1))>\epsilon e e^n\right]\\
&=4\left(1-\Phi\left(\frac{\epsilon e
e^n}{\bar{v}(n+1)}\right)\right) =4\left(1-\Phi\left(\frac{\epsilon
e}{\sqrt{e^{-2n}\bar{v}^2(n+1)}}\right)\right).
\end{align*}
If we interpret $\Phi(\infty)=1$, this formula holds valid in the
case when $\bar{v}(n+1)=0$, because in this case $Z_0(n)=0$ a.s. Now
$e^{-2n}\bar{v}^2(n+1)=e^{-2n}\int_n^{n+1}
e^{2s}\|\sigma(s)\|^2_F\,ds\leq e^2 \theta^2(n).$
Since $\Phi$ is increasing, we have
\[
\mathbb{P}[Z_0(n)>\epsilon e] = 4\left( 1-\Phi\left(\frac{\epsilon
e}{\sqrt{e^{-2n}\tau(n+1)}}\right)\right)\leq 4\left( 1-
\Phi\left(\frac{\epsilon e}{e\theta(n)}\right)\right),
\]
so
\begin{equation} \label{eq.Z0nepsprob}
\mathbb{P}[Z_0(n)>\epsilon e] \leq 4\left(1-
\Phi\left(\frac{\epsilon}{\theta(n)}\right)\right).
\end{equation}
Therefore by \eqref{eq.S1mixed}, \eqref{eq.intsignto0}, \eqref{eq.mills} and \eqref{eq.Z0nepsprob} we
have $\sum_{n=1}^\infty \mathbb{P}[Z_0(n)>\epsilon e]<+\infty$ for all
$\epsilon>\epsilon'$. Therefore by the Borel--Cantelli Lemma, we
have that
\begin{equation}\label{eq.Z0nbound}
\limsup_{n\to\infty}  Z_0(n)\leq \epsilon e, \quad \text{a.s.}
\end{equation}
By \eqref{eq.Y0bound}, \eqref{eq.supY0Y0nZ0n} and
\eqref{eq.Z0nbound} we have
\[
\limsup_{n\to\infty} \max_{t\in [n,n+1]} |Y_0(t)| \leq
\limsup_{n\to\infty} |Y_0(n)| + \limsup_{n\to\infty} Z_0(n) \leq
\frac{1}{1-e^{-1}}\epsilon + e \epsilon,
\]
Therefore, letting $\epsilon\downarrow \epsilon'$ through the
rational numbers we have
\[
\limsup_{t\to\infty} |Y_0(t)| \leq (1/(1-e^{-1})+e)\epsilon'=:c_1,
\quad \text{a.s.},
\]
proving the result.
\end{proof}

Before proceeding with the final supporting lemma, we show that
whenever $S_h'(\epsilon)$ is finite, we must have
\begin{equation} \label{eq.intsigttp10}
\lim_{t\to\infty} \int_t^{t+1} \|\sigma(s)\|^2_F\,ds=0.
\end{equation}
\begin{lemma}
Suppose that $S_h'$ obeys \eqref{eq.thetaboundedh}. Then $\sigma$ obeys
\eqref{eq.intsigttp10}.
\end{lemma}
\begin{proof}
By \eqref{eq.thetaboundedh}, there exists $\epsilon>0$ such that
$S_h'(\epsilon)<+\infty$. 
Therefore, it follows that $\int_{nh}^{(n+1)h} \|\sigma(s)\|^2_F\,ds\to 0$ as $n\to\infty$. This implies 
\[
\lim_{n\to\infty}\int_{n}^{n+1} \|\sigma(s)\|^2_F\,ds= 0.
\] 
For every $t>0$, there exists $n(t)\in\mathbb{N}$ such that $n(t)\leq
t<n(t)+1$. Hence
\begin{align*}
\int_t^{t+1}\|\sigma(s)\|^2_F\,ds &\leq \int_{n(t)}^{t+1}
\|\sigma(s)\|^2_F\,ds
=\int_{n(t)}^{n(t)+1} \|\sigma(s)\|^2_F\,ds+ \int_{n(t)+1}^{t+1} \|\sigma(s)\|^2_F\,ds\\
&\leq \int_{n(t)}^{n(t)+1} \|\sigma(s)\|^2_F\,ds+ \int_{n(t)+1}^{n(t)+2} \|\sigma(s)\|^2_F\,ds. 
\end{align*}
Since $n(t)\to\infty$ as $t\to\infty$ and  $\int_{n}^{n+1} \|\sigma(s)\|^2_F\,ds\to 0$  as
$n\to\infty$, taking limits yields \eqref{eq.intsigttp10}.
\end{proof}

Before we can show that $W$ is bounded, we must first prove that
\begin{equation}\label{eq.Zfiniteliminf}
\liminf_{t\to\infty} Z(t)<+\infty,\quad\text{a.s.}
\end{equation}
\begin{lemma} \label{lemma.Znottoinfty}
Suppose that $f$ obeys \eqref{eq.fglobalunperturbed} 
and \eqref{eq.fasybounded}. Suppose that $\sigma$ 
obeys \eqref{eq.sigmacns} and that $S_h'$ obeys
\eqref{eq.thetaboundedh}. Suppose that $X$ is a continuous adapted process 
which obeys \eqref{eq.mainstocheqn}. Let $Z$ be the unique continuous adapted process corresponding to $X$ which obeys \eqref{def.ZcompSDE}.  
Then $Z$ also obeys \eqref{eq.Zfiniteliminf}.
\end{lemma}
\begin{proof}
Note that if $f$ obeys \eqref{eq.fasybounded}, then by Lemma~\ref{lemma.phi1loclip} (specifically \eqref{eq.phi1pos}), $\phi_1$ given by
\eqref{def.phi1} satisfies $\lim_{x\to\infty} \phi_1(x)=+\infty$.
Using \eqref{def.ZcompSDE}, we have
\begin{equation} \label{eq.Zboundedmaster}
\frac{Z(t)}{t}=\frac{1+\|\xi\|}{t} -\frac{1}{t}\int_0^t
\phi_1(Z(s))\,ds + \frac{1}{t}\int_0^t
\frac{\|\sigma(s)\|^2_F+e^{-s}}{Z(s)}\,ds + \frac{M_2(t)}{t},
\end{equation}
where $M_2$ is the continuous martingale given by
\[
M_2(t)=\sum_{j=1}^r \int_0^t \bar{\sigma}_j(s)\,dB_j(s), \quad
\text{a.s.}
\]
Using \eqref{def.barsigma} we get
\[
\langle M_2\rangle(t)= \int_0^t \bar{\sigma}^2(s)\,ds\leq \int_0^t
\|\sigma(s)\|_F^2\,ds,
\]
and in the case when $S_h'(\epsilon)$ is finite, we may appeal to the
proof of Theorem~\ref{theorem.Xiffsigma}, which shows that
\eqref{eq.cesarosigma0} holds. On the event $A$ for which $\langle
M_2\rangle(t)$ tends to a finite limit as $t\to\infty$, we have that
$M_2(t)$ converges to a finite limit, in which case $M_2(t)/t\to 0$
as $t\to\infty$ on $A$. On $\bar{A}$, we have that $\langle
M_2\rangle(t)\to\infty$ as $t\to\infty$, so by the strong law of
large numbers for martingales, we have
\begin{align*}
\limsup_{t\to\infty} \frac{|M_2(t)|}{t}&\leq \limsup_{t\to\infty}
\frac{M_2(t)}{\langle M_2\rangle(t)}\limsup_{t\to\infty}
\frac{\langle M_2\rangle(t)}{t}\\
&= \limsup_{t\to\infty} \frac{M_2(t)}{\langle
M_2\rangle(t)}\limsup_{t\to\infty} \frac{1}{t}\int_0^t
\|\sigma(s)\|^2_F\,ds=0,
\end{align*}
so a.s. we have
\begin{equation} \label{eq.M2divt0}
\lim_{t\to\infty} \frac{M_2(t)}{t}=0,\quad\text{a.s.}
\end{equation}
Now define the event $A_1$ by  $A_1:=\{\omega:\lim_{t\to\infty}
Z(t,\omega)=\infty\}$ and suppose that $\mathbb{P}[A_1]>0$. By
Lemma~\ref{lemma.Zpos} we note that there is an a.s. event
$\Omega_3=\{\omega: Z(t,\omega)>0\text{ for all $t\geq 0$}\}$. Let
$A_2=A_1\cap \Omega_1\cap \Omega_2$, where $\Omega_1$ is the a.s.
event in \eqref{eq.M2divt0}. Thus $\mathbb{P}[A_2]>0$. Then for each
$\omega\in A_2$, we have that $\lim_{t\to\infty}
\phi_1(Z(t,\omega))=+\infty$, and so
\begin{equation} \label{eq.phiZdivt0}
\lim_{t\to\infty} \frac{1}{t}\int_0^t \phi_1(Z(s))\,ds=+\infty, \quad
\text{on $A_2$}.
\end{equation}
For each $\omega\in A_2$, there is a $T^\ast(\omega)>0$ such that
$Z(t,\omega)\geq 1$ for all $t\geq T^\ast(\omega)$. Therefore, for
$t\geq T^\ast(\omega)$, we have the bound
\[
\frac{1}{t}\int_0^t \frac{\|\sigma(s)\|^2_F+e^{-s}}{Z(s)}\,ds \leq
\frac{1}{t}\int_0^{T^\ast} \frac{\|\sigma(s)\|^2_F+e^{-s}}{Z(s)}\,ds
+ \frac{1}{t}\int_{T^\ast}^t \{\|\sigma(s)\|^2_F+e^{-s}\}\,ds.
\]
Since $t\mapsto e^{-t}$ is integrable, and $\sigma$ obeys
\eqref{eq.cesarosigma0}, it follows that the second term on the
right--hand side has a zero limit as $t\to\infty$. To deal with the
first term, note that the continuity of $Z$ on the compact interval
$[0,T^\ast]$ and the positivity of $Z$ implies there is a
$T_1^\ast\in [0,T^\ast]$ such that $\inf_{t\in [0,T^\ast]}
Z(t)=Z(T_1^\ast)>0$, and so the first term also tends to zero as
$t\to\infty$. Thus the third term on the righthand side of
\eqref{eq.Zboundedmaster} tends to zero as $t\to\infty$ on $A_2$.
Noting this zero limit, we take the limit as $t\to\infty$ in
\eqref{eq.Zboundedmaster}, and using \eqref{eq.phiZdivt0} and
\eqref{eq.M2divt0}, arrive at
\[
\lim_{t\to\infty} \frac{Z(t,\omega)}{t}=-\infty, \text{ for each
$\omega\in A_2$.}
\]
which implies that $Z(t,\omega)\to-\infty$ as $t\to\infty$ for each
$\omega\in A_2$. But since $Z(t,\omega)>0$ for all $t\geq 0$ for
each $\omega\in A_2$, we have a contradiction, proving the result.
\end{proof}
Finally, we are in a position to show that the process $W$ defined as the unique solution of the random differential equation \eqref{def.WcompSDE} 
corresponding to a solution $X$ of \eqref{eq.mainstocheqn}, is bounded by a deterministic constant almost surely.  
\begin{lemma} \label{lemma.Wbounded}
Suppose that $f$ obeys \eqref{eq.fglobalunperturbed} 
and \eqref{eq.fasybounded}. Suppose that $\sigma$
obeys \eqref{eq.sigmacns} and that $S_h'$ obeys
\eqref{eq.thetaboundedh}. Suppose that $X$ is a continuous adapted process obeying \eqref{eq.mainstocheqn}. 
Let $W$ be the unique continuous adapted process corresponding to $X$ which obeys \eqref{def.WcompSDE}. Then there is a 
deterministic $c_2>0$ such that 
\[
\limsup_{t\to\infty} |W(t)|\leq c_2, \quad \text{a.s.}
\]
\end{lemma}
\begin{proof}
We have by Lemma~\ref{lemma.Y0bounded} that $\limsup_{t\to\infty}
|Y_0(t)|\leq c_1$, a.s. From this fact and \eqref{eq.intsigttp10},
it follows that for every $\epsilon>0$ there exists a
$T(\omega,\epsilon)>0$ such that
\begin{equation}\label{eq.Y0boundintbound}
|Y_0(t,\omega)|\leq c_1+1:=\bar{Y}, \quad \int_{t-1}^t
\left\{\|\sigma(s)\|^2_F + e^{-s}\right\}\,ds<1, \quad t\geq
T(\epsilon,\omega).
\end{equation}
Suppose this holds on the a.s. event $\Omega_1$. By \eqref{eq.phi1pos} and \eqref{def.phi1} 
we have that
$\phi_1(x)\to\infty$ as $x\to\infty$. Therefore, we can choose $M>0$
so large that
\begin{equation} \label{eq.Mconditions}
\frac{M}{2}\geq 2\bar{Y}+1, \quad \inf_{x\geq M/2-\bar{Y}}
\phi_1(x)>\frac{1}{\bar{Y}+1}+\bar{Y}+1.
\end{equation}

By Lemma~\ref{lemma.Znottoinfty}, there is an a.s. event $\Omega_2$
such that $\Omega_2=\{\omega:\liminf_{t\to\infty}
Z(t,\omega)<+\infty\}$. Since $|Y_0|$ has a finite limsup on
$\Omega_2$, if follows that $\liminf_{t\to\infty}
\|W(t,\omega)\|<+\infty$ on $\Omega_1\cap\Omega_2$. Next suppose
there is an event $A_3=\{\omega:\limsup_{t\to\infty}
W(t,\omega)>M\}$ for which $\mathbb{P}[A_3]>0$. Let
$A_4=A_3\cap\Omega_2\cap\Omega_3$. Notice that
$\liminf_{t\to\infty} W(t)=\liminf_{t\to\infty} Z(t)+Y_0(t)\geq
\liminf_{t\to\infty} Y_0(t)\geq -c_1$, so we do not need to consider
the absolute value of $W$ in the definition of $A_3$. Suppose that
$\omega\in A_4$. It then follows that there exists $t_1>T(\epsilon)$
such that $t_1=\inf\{t>T(\epsilon): W(t)=M/2\}$ and a $t_2>t_1$ such
that $t_2=\inf\{t>t_1: W(t)=M\}$. It also follows that there is
$t_1'\in [t_1,t_2)$ such that $t_1'=\sup\{t>t_1:W(t)=M/2\}$.

Suppose first that $t_2-t_1'\geq 1$. Then $t_2-1\geq t_1'\geq
t_1>T(\epsilon)$. Define $t_3=t_2-1$. Then $M>W(t_3)>M/2$. Hence
\begin{align*}
M-W(t_3)&=W(t_2)-W(t_3)\\
&=-\int_{t_2-1}^{t_2}\phi_1(W(s)+Y_0(s))\,ds + \int_{t_2-1}^{t_2}
\left\{\frac{e^{-s}+\|\sigma(s)\|^2_F}{W(s)+Y_0(s)} +
Y_0(s)\right\}\,ds.
\end{align*}
Since $W(t)>M/2$ and $|Y_0(t)|\leq \bar{Y}$ for all $t\in
[t_2-1,t_2]$, we have that $W(t)+Y_0(t)\geq M/2-\bar{Y}>0$. Thus
$\phi_1(W(t)+Y(t))\geq \inf_{x\geq M/2-\bar{Y}} \phi_1(x)$. Using these
estimates leads to
\begin{align*}
M-W(t_3) &\leq -\int_{t_2-1}^{t_2} \inf_{x\geq M/2-\bar{Y}} \phi_1(x)
\,ds
+ \int_{t_2-1}^{t_2} \left\{\frac{e^{-s}+\|\sigma(s)\|^2_F}{M/2-\bar{Y}} + \bar{Y}\right\}\,ds\\
&= -\inf_{x\geq M/2-\bar{Y}} \phi_1(x) + \frac{1}{M/2-\bar{Y}}
\int_{t_2-1}^{t_2} \left\{e^{-s}+\|\sigma(s)\|^2_F\right\}\,ds +
\bar{Y}.
\end{align*}
Using the fact that $t_2-1>T(\epsilon)$, we may use the second
condition in \eqref{eq.Y0boundintbound}, the first condition in
\eqref{eq.Mconditions} and then the  last condition in
\eqref{eq.Mconditions} to get
\begin{align*}
0<M-W(t_3) &\leq -\inf_{x\geq M/2-\bar{Y}} \phi_1(x)
+ \frac{1}{M/2-\bar{Y}} + \bar{Y}\\
&\leq -\inf_{x\geq M/2-\bar{Y}} \phi_1(x) + \frac{1}{\bar{Y}+1} +
\bar{Y}<0,
\end{align*}
a contradiction.

Suppose on the other hand that $t_2-t_1'<1$. Once again, for all
$t\in (t_1',t_2)$ we have $M/2<W(t)<M$ with $W(t_1')=M/2$ and
$W(t_2)=M$. Then, as $\phi_1(x)\geq 0$ for all $x\geq 0$, we have
\begin{align*}
M/2&=W(t_2)-W(t_1')\\
&=-\int_{t_1'}^{t_2} \phi_1(Z(s))\,ds + \int_{t_1'}^{t_2} \frac{e^{-s}+\|\sigma(s)\|^2_F}{W(s)+Y_0(s)}\,ds + \int_{t_1'}^{t_2} Y_0(s)\,ds\\
&\leq \int_{t_1'}^{t_2}
\frac{e^{-s}+\|\sigma(s)\|^2_F}{W(s)+Y_0(s)}\,ds + \int_{t_1'}^{t_2}
|Y_0(s)|\,ds.
\end{align*}
Now, for all $t\in [t_1',t_2]$ we have that $W(t)\geq M/2$ and
$|Y_0(t)|\leq \bar{Y}$, so $W(t)+Y_0(t)\geq M/2-\bar{Y}>0$. Using
these estimates, and then the assumption that $t_2-t_1'<1$, we get
\begin{align*}
M/2
&\leq \int_{t_1'}^{t_2} \frac{e^{-s}+\|\sigma(s)\|^2_F}{W(s)+Y_0(s)}\,ds + \int_{t_1'}^{t_2} |Y_0(s)|\,ds\\
&\leq \frac{1}{M/2-\bar{Y}}\int_{t_1'}^{t_2} \{ e^{-s}+\|\sigma(s)\|^2_F\}\,ds + \int_{t_1'}^{t_2} \bar{Y}\,ds\\
&\leq \frac{1}{M/2-\bar{Y}}\int_{t_2-1}^{t_2} \{
e^{-s}+\|\sigma(s)\|^2_F\}\,ds + \bar{Y}.
\end{align*}
Finally, we notice that $t_2> t_1> T(\epsilon)$, so we may use the
second estimate in \eqref{eq.Y0boundintbound} to get $M/2 \leq
1/(M/2-\bar{Y}) + \bar{Y}$. Since $M/2>\bar{Y}$, this rearranges to
give $(M/2-\bar{Y})^2\leq 1$ or $M/2-\bar{Y}\leq 1$. This is
$M/2\leq \bar{Y}+1$. But as $\bar{Y}>0$, this contradicts the second
condition in  \eqref{eq.Mconditions}, i.e, $M/2\geq 2\bar{Y}+1$.
 \end{proof}
The proof of the upper bound on $\limsup_{t\to\infty}\|X(t)\|$ in part (B) of Theorem~\ref{theorem.Xclassify} is now immediate. 
It follows from Lemma
\ref{lemma.XZcomp} that
\[
\|X(t)\|\leq Z(t)=W(t)+Y_0(t), \quad t\geq 0,
\]
where $W$ and $Y_0$ are given by \eqref{def.WcompSDE} and
\eqref{def.Y0} respectively. By Lemma~\ref{lemma.Y0bounded}, we have
that $\limsup_{t\to\infty}\|Y_0(t)\|\leq c_1$ a.s. Also by
Lemma~\ref{lemma.Wbounded}, we may conclude that 
$\limsup_{t\to\infty}\|W(t)\|\leq c_2$ a.s. Notice that both $c_1$
and $c_2$ are deterministic bounds. Therefore, it follows that
\[
\limsup_{t\to\infty} \|X(t)\|\leq c_1+c_2, \quad \text{a.s.},
\]
as required.

\subsection{Proof that limit inferior is zero in part (B) of Theorem~\ref{theorem.Xclassify}}
It remains to prove the second part of (B) in Theorem~\ref{theorem.Xclassify}, namely that
\[
\liminf_{t\to\infty} \|X(t)\|=0, \quad\text{a.s.}
\]
We have already shown that $t\mapsto \|X(t)\|$ is bounded.
Furthermore, since $S_h'(\epsilon)<+\infty$ for all
$\epsilon>\epsilon'$, we can prove as in the proof of
Theorem~\ref{theorem.Xiffsigma} that \eqref{eq.cesarosigma0}
holds i.e.,
\[
\lim_{t\to\infty} \frac{1}{t}\int_0^t \|\sigma(s)\|^2_F\,ds=0.
\]
 Recall from \eqref{eq.Xforboundedtoo} that we have the
representation
\begin{equation*}
\|X(t)\|^2=\|\xi\|^2 - \int_0^t 2\langle
X(s),f(X(s))\rangle\,ds+\int_0^t \|\sigma(s)\|^2_F\,ds + M(t),
\quad t\geq 0,
\end{equation*}
where $M$ is the local (scalar) martingale given by
\eqref{eq.Mforboundedtoo} i.e.,
\begin{equation*}
M(t)=2\sum_{j=1}^r \int_0^t \sum_{i=1}^d
X_i(s)\sigma_{ij}(s)\,dB_j(s), \quad t\geq 0.
\end{equation*}
The quadratic variation of $M$ is given by
\[
\langle M\rangle (t)=4\sum_{j=1}^r \int_0^t \left(\sum_{i=1}^d
X_i(s)\sigma_{ij}(s)\right)^2\,ds,
\]
and so by the Cauchy--Schwarz inequality, we have
\[
\langle M\rangle (t)\leq 4\sum_{j=1}^r \int_0^t \sum_{i=1}^d
X_i^2(s)\sum_{i=1}^d\sigma_{ij}^2(s) \,ds \leq   4\int_0^t
\|X(s)\|_2^2 \|\sigma(s)\|^2_F \,ds.
\]
Therefore, as $t\mapsto \|X(t)\|$ is a.s. bounded, we have
\[
\lim_{t\to\infty} \frac{1}{t}\langle M\rangle(t)=0, \quad
\text{a.s.}
\]
In the case that $\langle M\rangle$ converges, we have that $M$
tends to a finite limit and so
\[
\lim_{t\to\infty} \frac{1}{t}M(t)=0.
\]
If, on the other hand  $\langle M\rangle(t)\to\infty$ as
$t\to\infty$, by the strong law of large numbers for martingales, we
have
\[
\lim_{t\to\infty} \frac{1}{t}M(t)=\lim_{t\to\infty}
\frac{M(t)}{\langle M\rangle(t)}\cdot \frac{\langle
M\rangle(t)}{t}=0.
\]
Using the fact that  $t\mapsto \|X(t)\|$ is bounded, we have
$\|X(t)\|^2/t\to 0$ as $t\to\infty$. Therefore, by rearranging
\eqref{eq.Xforboundedtoo}, dividing by $t$ and letting $t\to\infty$,
we get \eqref{eq.aveXfXintto0}, as claimed in part (B) of Theorem~\ref{theorem.Xclassify}.

To show that the liminf is zero a.s., we suppose to the contrary that there is an event $A_1$ of positive probability
such that
\[
A_1=\{\omega:\liminf_{t\to\infty} \|X(t,\omega)\|>0\}.
\]
Since $X$ is bounded, it follows that for a.a. $\omega\in A_1$,
there are $\bar{X}(\omega), \bar{x}(\omega)\in (0,\infty)$ such that
\[
\liminf_{t\to\infty} \|X(t,\omega)\|=\bar{x}(\omega), \quad
\limsup_{t\to\infty} \|X(t,\omega)\|=\bar{X}(\omega).
\]
Thus, there exists $T(\omega)>0$ such that
\[
\frac{\bar{x}(\omega)}{2}\leq \|X(t,\omega)\|\leq 2\bar{X}(\omega),
\quad t\geq T(\omega).
\]
By the continuity of $f$ and the fact that $\langle x,f(x)\rangle
>0$ for all $x\neq 0$, it follows that for any $0<a\leq b<+\infty$
\[
\inf_{\|x\|\in [a,b]} \langle x,f(x)\rangle = L(a,b)>0.
\]
Hence for $t\geq T(\omega)$ we have
\[
\langle X(t,\omega),f(X(t,\omega))\rangle \geq
L\left(\frac{\bar{x}(\omega)}{2},
2\bar{X}(\omega)\right)=:\lambda(\omega)>0.
\]
Hence for $t\geq T(\omega)$ we have
\[
\frac{1}{t} \int_0^t \langle X(s,\omega),f(X(s,\omega))\rangle\,ds
\geq \frac{1}{t}\int_{T(\omega)}^t \langle
X(s,\omega),f(X(s,\omega))\rangle\,ds \geq \frac{t-T(\omega)}{t}
\cdot \lambda(\omega).
\]
Hence for a.a. $\omega\in A_1$ we have
\[
\liminf_{t\to\infty} \frac{1}{t} \int_0^t \langle
X(s,\omega),f(X(s,\omega))\rangle\,ds \geq  \lambda(\omega)>0,
\]
which implies that
\[
\liminf_{t\to\infty} \frac{1}{t} \int_0^t \langle
X(s),f(X(s))\rangle\,ds >0, \quad\text{a.s. on $A_1$}.
\]
This limit, taken together with the fact that $A_1$ is an event of
positive probability, contradicts \eqref{eq.aveXfXintto0}. Hence, it
must follow that $\mathbb{P}[A_1]=0$. This implies that
$\mathbb{P}[\overline{A}_1]=1$, or that $\liminf_{t\to\infty}
\|X(t)\|=0$ a.s. as required.


\begin{thebibliography}{10}
\bibitem{JAJCAR:2011dresden}
J.A.D.Appleby, J. Cheng and A.Rodkina, Characterisation of the
asymptotic behaviour of scalar linear differential equations with
respect to a fading stochastic perturbation, \emph{Discrete. Contin.
Dynam. Syst.}, Suppl., 79--90, 2011.

\bibitem{JAJC:2011szeged}
J.~A.~D.~Appleby and J.~Cheng, On the asymptotic stability of a
class of perturbed ordinary differential equations with weak
asymptotic mean reversion, \emph{E. J. Qualitative Theory of Diff.
Equ.}, Proc. 9th Coll., No. 1 (2011), pp. 1-36.

\bibitem{JAAR:2010a}
J.~A.~D. Appleby, J. Cheng and A.~Rodkina. On the Classification of
the Asymptotic Behaviour of Solutions of Globally Stable Scalar
Differential Equations with Respect to State--Independent Stochastic
Perturbations, preprint 2012.

\bibitem{JAJCAR:2012}
J.~A.~D.~Appleby, J.~Cheng and A.~Rodkina. Classification of the Asymptotic Behaviour of Globally Stable Linear Differential Equations with Respect to
State--independent Stochastic Perturbations, preprint 2012.

\bibitem{JAJGAR:2009}
J.~A.~D. Appleby, J.~G.~Gleeson and A.~Rodkina. On asymptotic
stability and instability with respect to a fading stochastic
perturbation, \emph{Applicable Analysis}, 88 (4), 579--603, 2009.

\bibitem{JAARMR:2009}
J.~A.~D.~Appleby,  M.~Riedle and A.~Rodkina. On Asymptotic Stability
of linear stochastic Volterra difference equations with respect to a
fading perturbation, \emph{Adv. Stud. Pure Math.}, 53, 271--282,
2009.

\bibitem{Chan:1989}
T.~Chan. On multi--dimensional annealing problems, \emph{Math. Proc.
Camb. Philos. Soc.}, 105, 177-�184, 1989.

\bibitem{ChanWill:1989}
T.~Chan and D.~Williams. An ``excursion'' approach to an annealing
problem, \emph{Math. Proc. Camb. Philos. Soc.}, 105, 169--176, 1989.

\bibitem{KarShr:90}
I. Karatzas and S.E. Shreve, Brownian Motion and Stochastic
Calculus, 2nd edition, Springer, New York, 1991.

\bibitem{LipShir:1989}
R. Sh. Liptser and A. N. Shiryaev, Theory of Martingales, Kluwer
Academic Publishers, Dordrecht, 1989.

\bibitem{Mao1}
X.~Mao.
\newblock {\em Stochastic Differential Equations and their Applications}.
\newblock Horwood Publishing Limited, Chichester, 1997.

\bibitem{RevYor}
\newblock D.~Revuz and M.~Yor,
\newblock ``Continuous Martingales and Brownian Motion," 3$^{rd}$
\newblock edition,
\newblock Springer-Verlag, New York, 1999.
%

\bibitem{RogWill:1989} Rogers, L.C.G and Williams, D. \emph{Diffusions, Markov Processes
and Martingales}, Cambridge Mathematical Library, 2000.

\bibitem{hart:1961}
P.~Hartman, On stability in the large for systems of ordinary
differential equations. \emph{Canad. J. Math.}, 13, 1961, 480--492.

\bibitem{hartolech:1962}
P.~Hartman and C.~Olech, On global asymptotic stability of solutions
of differential equations. \textit{Trans. Amer. Math. Soc.}, 104,
1962, 154--178.

\bibitem{olech:1960}
C.~Olech, On the global stability of an autonomous system on the
plane. \emph{Contributions to Differential Equations}, 1, 1963,
389--400.

\bibitem{gasullllibresoto:1991}
A.~Gasull, J.~Llibre, and J.~Sotomayor, Global asymptotic stability
of differential equations in the plane. \emph{J. Differential
Equations}, 91, 1991, no. 2, 327--335.

\bibitem{brockschein:1976}
W.~A.~Brock and J.~A.~Scheinkman, Global asymptotic stability of
optimal control systems with applications to the theory of economic
growth. Hamiltonian dynamics in economics. \emph{J. Econom. Theory},
12, 1976, no. 1, 164--190.


%
\end{thebibliography}
\end{document}